\documentclass[a4paper,12pt]{amsart}
\usepackage[utf8]{inputenc}
\usepackage[english]{babel}
\usepackage{amsmath, amssymb, latexsym, amsfonts}
\usepackage[matrix,arrow,curve]{xy}
\usepackage[T2A]{fontenc}
\usepackage{tikz}
\usetikzlibrary{arrows,matrix,fit,positioning}

\newtheorem{theorem}{Theorem}
\newtheorem*{thmIntr}{Theorem}

\newtheorem{lemma}{Lemma}
\newtheorem{proposition}{Proposition}
\newtheorem{cor}{Corollary}
\theoremstyle{definition}
\newtheorem{definition}{Definition}

\theoremstyle{remark}
\newtheorem{rem}{Remark}

\newcommand{\SL}{{\rm SL}}

\newcommand{\pt}{{\rm pt}}
\newcommand{\Proj}{\mathbb{P}}

\newcommand{\Tc}{\mathcal{T}}
\newcommand{\SF}{\mathcal{SF}}
\newcommand{\A}{\mathbb{A}}
\newcommand{\SSp}{\mathbb{S}}
\newcommand{\PP}{\mathbb{P}}
\newcommand{\rank}{\operatorname{rank}}
\newcommand{\triv}{\mathcal{O}}
\newcommand{\Gm}{{\mathbb{G}_m}}
\newcommand{\SH}{\mathcal{SH}}

\newcommand{\BO}{\mathbf{BO}}
\newcommand{\unit}{e}
\newcommand{\Spec}{\operatorname{Spec}}

\newcommand{\bigslant}[2]{{\left.\raisebox{.2em}{$#1$}\middle/\raisebox{-.2em}{$#2$}\right.}}

\begin{document}

\title{
The special linear version of the projective bundle theorem
}
\author{Alexey Ananyevskiy}

\begin{abstract}
A special linear Grassmann variety $SGr(k,n)$ is the complement to the zero section of the determinant of the tautological vector bundle over $Gr(k,n)$. For an $SL$-oriented representable ring cohomology theory $A^*(-)$ with invertible stable Hopf map $\eta$, including Witt groups and $MSL_\eta^{*,*}$, we have $A^*(SGr(2,2n+1))\cong A^*(pt)[e]/\big( e^{2n}\big)$, and $A^*(SGr(k,n))$ is a truncated polynomial algebra over $A^*(pt)$ whenever at least one of the integers $k,n-k$ is even. A splitting principle for such theories is established. Using the computations for the special linear Grassmann varieties we obtain a description of $A^*(BSL_n)$ in terms of homogeneous power series in certain characteristic classes of tautological bundles.
\end{abstract}

\maketitle

\section{Introduction.}

The basic and most fundamental computation for oriented cohomology theories is the projective bundle theorem (see \cite{Mor1} or \cite[Theorem~3.9]{PS}) claiming $A^*(\PP^n)$ to be a truncated polynomial ring over $A^*(pt)$ with an explicit basis in terms of the powers of a Chern class. Having this result at hand one can define higher characteristic classes and compute the cohomology of Grassmann and flag varieties. In particular, the fact that cohomology of a full flag variety is a truncated polynomial algebra gives rise to a splitting principle, which states that from a viewpoint of an oriented cohomology theory every vector bundle is in a certain sense a sum of linear bundles. For a representable cohomology theory one can deal with an infinite dimensional Grassmannian which is a model for the classifying space $BGL_n$ and obtain even neater answer, the formal power series in the characteristic classes of the tautological vector bundle.

There are analogous computations for symplectically oriented cohomology theories \cite{PW1} with appropriately chosen varieties: quaternionic projective spaces $HP^n$ instead of the ordinary ones and quaternionic Grassmannian and flag varieties. The answers are essentially the same, algebras of truncated polynomials in characteristic classes. In loc. cit. these classes were referred to as ``Pontryagin classes'', but it was noted by Buchstaber that it would be more convenient to name them ``Borel classes'' since they correspond to the symplectic Borel classes in topology. We prefer to adopt this modification of the terminology.

These computations have a variety of applications, for example theorems of Conner and Floyd's type \cite{CF} describing $K$-theory and hermitian $K$-theory as quotients of certain universal cohomology theories \cite{PPR1,PW4}.

In the present paper we establish analogous results for $SL$-oriented cohomology theories. The notion of such orientation was introduced in \cite[Definition~5.1]{PW3}. At the same preprint there was constructed a universal $SL$-oriented cohomology theory, namely the algebraic special linear cobordisms $MSL$, \cite[Definition~4.2]{PW3}. A more down to earth example is derived Witt groups defined by Balmer \cite{Bal1} and oriented via Koszul complexes \cite{Ne2}. A comprehensive survey on the Witt groups could be found in \cite{Bal2}. Of course, every oriented cohomology theory admits a special linear orientation, but it will turn out that we are not interested in such examples. We will deal with representable cohomology theories and work in the unstable $H_\bullet(k)$  and stable $\SH(k)$ motivic homotopy categories introduced by Morel and Voevodsky \cite{MV,V}. We recall all the necessary constructions and notions in sections 2-4 as well as provide preliminary calculations with special linear orientations.

We need to choose an appropriate version of "projective space" analogous to $\PP^n$ and $HP^n$. Natural candidates are $SL_{n+1}/SL_{n}$ and $\A^{n+1}-\{0\}$. There is no difference which one to choose since the first one is an affine bundle over the latter one, so they have the same cohomology. We take $\A^{n+1}-\{0\}$ since it looks prettier  from the geometric point of view. There is a calculation for the Witt groups of this space \cite[Theorem~8.13]{BG} claiming that $W^*(\A^{n+1}-\{0\})$ is a free module of rank two over $W^*(pt)$ with an explicit basis. The fact that it is a free module of rank two is not surprising since $\A^{n+1}-\{0\}$ is a sphere in the stable homotopy category $\SH_\bullet(k)$ and $W^*(-)$ is representable \cite{Hor}. The interesting part is the basis. Let $\Tc=\triv^{n+1}/\triv(-1)$ be the tautological $\rank$ $n$ bundle over $\A^{n+1}-\{0\}$. Then for $n=2k$ the basis consists of the element $1$ and the class of a Koszul complex. The latter one is the Euler class $e(\Tc)$ in the Witt groups. Unfortunately, for the odd $n$ the second term of the basis looks more complicated. Moreover, for an oriented cohomology theory even in the case of $n=2k$ the corresponding Chern class vanishes, so one can not expect that $1$ and $e(\Tc)$ form a basis for every cohomology theory with a special linear orientation.

Here comes into play the following observation. The maximal compact subgroup of $SL_n(\mathbb{R})$ is $SO_n(\mathbb{R})$, so over $\mathbb{R}$ the notion of a special linear orientation of a vector bundle derives to the usual topological orientation of a bundle. The Euler classes of oriented vector bundles in topology behave themselves well only after inverting $2$ in the coefficients, so we want to invert in the algebraic setting something analogous to $2$. There are two interesting elements in the stable cohomotopy groups $\pi^{*,*}(pt)$ that go to $2$ after taking $\mathbb{R}$-points, a usual $2\in \pi^{0,0}(pt)$ and the stable Hopf map $\eta\in \pi^{-1,-1}(pt)$ arising from the morphism $\A^2-\{0\}\to \PP^1$. In general $2$ is not invertible in the Witt groups, so we will invert $\eta$. Moreover, recall a theorem due to Morel \cite{Mor2} claiming that for a perfect field there is an isomorphism $\oplus_n\pi^{n,n}(\Spec k)[\eta^{-1}]\cong W^0(k)[\eta,\eta^{-1}]$, so in a certain sense $\eta$ is invertible in the Witt groups. In sections 5-6 we do some computations justifying the choice of $\eta$.

In this paper we deal mainly with the cohomology theories obtained as follows. Take a commutative monoid $(A,m,e: \SSp\to A)$ in the stable homotopy category $\SH(k)$ and fix a special linear orientation on the cohomology theory $A^{*,*}(-)$. The unit $e: \SSp\to A$ of the monoid $(A,m,e)$ induces a morphism of cohomology theories $\pi^{*,*}(-)\to A^{*,*}(-)$ making $A^{*,*}(X)$ an algebra over the stable cohomotopy groups. Inverting stable Hopf map $\eta$ we obtain a cohomology theory
\[
A^{*,*}_\eta(X)=A^{*,*}(X)[\eta^{-1}].
\]
This theory is periodic in the $(1,1)$--direction via cup-product $(-\cup\eta^n)$, thus without loss of generality we may focus on
\[
A^{*}(X)=A^{*,0}_\eta(X).
\]
It is still a cohomology theory. One can regard it as a $(1,1)$-periodic cohomology theory $A^{*,*}_\eta(-)$ collapsed in the $(1,1)$-direction. For these cohomology theories we have a result analogous to the case of the Witt groups.
\begin{thmIntr}
$
A^*(\A^{2n+1}-\{0\})= A^{*}(pt)\oplus A^{*-2n}(pt)e(\Tc).
$
\end{thmIntr}
\noindent The relative version of this statement is Theorem~\ref{thm_main} in section~7. Note that there is no similar result for $\A^{2n}-\{0\}$.

In the next section we consider another family of varieties, called special linear Grassmannians $SGr(2,n)=SL_{n}/P_2'$, where $P_2'$ stands for the derived group of the parabolic subgroup $P_2$, i.e. $P_2'$ is the stabilizer of the bivector $e_1\wedge e_2$ in the exterior square of the regular representation of $SL_n$. There are tautological bundles $\Tc_1$ and $\Tc_2$ over $SGr(2,n)$ of ranks $2$ and $n-2$ respectively. We have the following theorem which seems to be the correct version of the projective bundle theorem in the special linear setting.
\begin{thmIntr}
For the special linear Grassmann varieties we have the next equalities.
$$
A^*(SGr(2,2n))=\bigoplus\limits_{i=0}^{2n-2} A^{*-2i}(pt)e(\Tc_1)^i\oplus A^{*-2n+2}(pt)e(\Tc_2),
$$
$$
A^*(SGr(2,2n+1))= \bigoplus\limits_{i=0}^{2n-1} A^{*-2i}(pt)e(\Tc_1)^i.
$$
\end{thmIntr}
\noindent Recall that there is a recent computation of the twisted Witt groups of Grassmannians \cite{BC}. The twisted groups are involved since the authors use pushforwards that exist only in the twisted case. We deal with the varieties with a trivialized canonical bundle and closed embeddings with a special linear normal bundle in order to avoid these difficulties. In fact we are interested in the relative computations that could be extended to the Grassmannian bundles, so we look for a basis consisting of characteristic classes rather then pushforwards of certain elements. It turns out that such bases exist only for the special linear flag varieties with all but at most one dimension step being even, i.e. we can handle $SGr(1,7)$, $\SF(2,4,6)$ and $\SF(2,5,7)$ but not $SGr(3,6)$. Nevertheless it seems that one can construct the basis for the latter case in terms of pushforwards.

Sections 8 and 9 are devoted to the computations of the cohomology rings of partial flag varieties. We obtain an analogue of the splitting principle in Theorem~\ref{thm_split} and derive certain properties of the characteristic classes. In particular, there is a complete description of the cohomology rings of maximal $SL_2$ flag varieties,
\[
\SF(2n)=SL_{2n}/P_{2,4,\dots,2n-2}',\quad
\SF(2n+1)=SL_{2n+1}/P_{2,4,\dots,2n}'.
\]
The result looks as follows (see Remark~\ref{rem_SF}).
\begin{thmIntr}
For $n\ge 1$ consider
$$
s_i=\sigma_i(e_1^2,e_2^2,...,e_n^2),\quad t=\sigma_n(e_1,e_2,\dots,e_n)
$$
with $\sigma_i$ being the elementary symmetric polynomials in $n$ variables. Then we have the following isomorphisms
\begin{enumerate}
\item
$
A^*(\SF(2n))\cong \bigslant{A^*(pt)[e_1,e_2,...,e_{n}]}{\big( s_1,s_2,\dots,s_{n-1},t\big)},
$
\item
$
 A^*(\SF(2n+1))\cong \bigslant{A^*(pt)[e_1,e_2,...,e_{n}]}{\big( s_1,s_2,\dots,s_n\big)}.
$
\end{enumerate}
\end{thmIntr}
\noindent Note that one can substitute $SL_n/(SL_2)^{[n/2]}$ instead of $\SF(n)$. These answers and the choice of commuting $SL_2$ in $SL_n$ perfectly agree with our principle that $SL_n(\mathbb{R})$ stands for $SO_n(\mathbb{R})$, since $SL_2(\mathbb{R})$ stands for the compact torus $S^1\cong SO_2(\mathbb{R})$, and the choice of maximal number of commuting $SL_2$ is parallel to the choice of the maximal compact torus. We get the coinvariants for the Weyl groups $W(B_{n})$ and $W(D_n)$ and it is what one gets computing the cohomology of $SO_n(\mathbb{R})/T$. 

In section 12 we carry out a computation for the cohomology rings of the special linear Grassmannians $SGr(m,n)$ with at least one of the integers $m,n-m$ being even, see Theorem~\ref{thm_SGr}. The answer is a truncated polynomial algebra in certain characteristic classes. At the end we assemble the calculations for the special linear Grassmannians and compute in Theorem~\ref{thm_BSL} the cohomology of the classifying spaces $BSL_n$ in terms of homogeneous formal power series.
\begin{thmIntr}
We have the following isomorphisms.
$$
A^*(BSL_{2n})\cong A^*(pt)\left[\left[p_1,\dots,p_{n-1},e\right]\right]_h,
$$
$$
A^*(BSL_{2n+1})\cong A^*(pt)\left[\left[p_1,\dots,p_{n}\right]\right]_h.
$$
\end{thmIntr}

Finally, we leave for the forthcoming paper \cite{An} the careful proof of the fact that Witt groups arise from the hermitian $K$-theory in the described above fashion, that is $W^*(X)\cong (BO^{*,*}_\eta(X))^{*,0}$. In the same paper we obtain the following special linear version of the motivic Conner and Floyd theorem.
\begin{thmIntr}
Let $k$ be a field of characteristic different from $2$. Then for every smooth variety $X$ over $k$ there exists an isomorphism
\[
MSL^{*,*}_\eta(X)\otimes_{MSL^{4*,2*}(pt)}W^{2*}(pt)\cong W^{*}(X)[\eta,\eta^{-1}].
\]
\end{thmIntr}
Another application of the developed technique lies in the field of the equivariant Witt groups and we are going to address it in another paper.

{\it Acknowledgement.} The author wishes to express his sincere gratitude to I.~Panin for the introduction to the beautiful world of $\A^1$-homotopy theory and numerous discussions on the subject of this paper. This research is supported by RFBR grants 10-01-00551-a, 12-01-31100 and 12-01-33057 and by the Chebyshev Laboratory  (Department of Mathematics and Mechanics, St. Petersburg State University) under RF Government grant 11.G34.31.0026. 

\section{Preliminaries on $\SH(k)$ and ring cohomology theories.}
Let $k$ be a field of characteristic different from $2$ and let $Sm/k$ be the category of smooth varieties over $k$.

A motivic space over $k$ is a simplicial presheaf on $Sm/k$. Each $X\in Sm/k$ defines an unpointed motivic space $Hom_{Sm/k}(-,X)$ constant in the simplicial direction. We will often write $pt$ for the $\Spec k$ regarded as a motivic space. 

We use the injective model structure on the category of the pointed motivic spaces $M_\bullet(k)$.
Inverting the weak motivic equivalences in $M_\bullet(k)$ gives the pointed motivic unstable homotopy category $H_\bullet (k)$.

Let $T=\A^1/(\A^1-\{0\})$ be the Morel-Voevodsky object. A $T$-spectrum $M$ \cite{Jar} is a sequence of pointed motivic spaces $(M_0,M_1,M_2,\dots)$ equipped with the structural maps $\sigma_n  \colon T\wedge M_n\to M_{n+1}$. A map of $T$-spectra is a sequence of maps of pointed motivic spaces which is compatible with the structure maps. We write $MS(k)$ for the category of $T$-spectra. Inverting the stable motivic weak equivalences as in \cite{Jar} gives the motivic stable homotopy category $\SH(k)$.

A pointed motivic space $X$ gives rise to a suspension $T$-spectrum $\Sigma^{\infty}_T X$. Set $\SSp=\Sigma^{\infty}_T (pt_+)$ for the spherical spectrum. Both $H_\bullet(k)$ and $\SH(k)$ are equipped with symmetric monoidal structures $(\wedge,pt_+)$ and $(\wedge,\SSp)$ respectively and
\[
\Sigma^{\infty}_T\colon H_\bullet(k)\to \SH (k)
\]
is a strict symmetric monoidal functor. 

Recall that there are two spheres in $M_\bullet(k)$, the simplicial one $S^{1,0}=S_s^1= \Delta^1/\partial(\Delta^1)$ and $S^{1,1}=(\Gm,1)$. For the integers $p,q\ge 0$ we write $S^{p+q,q}$ for $(\Gm,1)^{\wedge q}\wedge (S_s^1)^{\wedge p}$ and $\Sigma^{p+q,q}$ for the suspension functor $-\wedge S^{p+q,q}$. This functor becomes invertible in the stable homotopy category $\SH(k)$, so we extend the notation to arbitrary integers $p,q$ in an obvious way. 

Any $T$-spectrum $A$ defines a bigraded cohomology theory on the category of pointed motivic spaces. Namely, for a pointed space $(X,x)$ one sets
$$
A^{p,q}(X,x)=Hom_{\SH(k)}(\Sigma^\infty_T (X,x),\Sigma^{p,q}A)
$$
and $A^{*,*}(X,x)=\bigoplus_{p,q}A^{p,q}(X,x)$. In case of $i-j,j\ge 0$ one has a canonical suspension isomorphism $A^{p,q}(X,x)\cong A^{p+i,q+j}(\Sigma^{i,j}(X,x))$ induced by the shuffling isomorphism $S^{p,q}\wedge S^{i,j}\cong S^{p+i,q+j}$. In the motivic homotopy category there is a canonical isomorphism $T\cong S^{2,1}$, we write 
\[
\Sigma_T\colon A^{*,*}(X)\xrightarrow{\simeq} A^{*+2,*+1}(X\wedge T)
\]
for the corresponding suspension isomorphism. See definition~\ref{def_rho} in Section~5 for the details. 

For an unpointed space $X$ we set $A^{p,q}(X)=A^{p,q}(X_+,+)$ with $A^{*,*}(X)$ defined accordingly. Set $\pi^{i,j}(X)=\SSp^{i,j}(X)$ to be the stable cohomotopy groups of $X$. 

We can regard smooth varieties as unpointed motivic spaces and obtain the groups $A^{p,q}(X)$. Given a closed embedding $i\colon Z\to X$ of varieties we write $Th(i)$ for $X/(X-Z)$. For a vector bundle $E\to X$ set $Th(E)=E/(E-X)$ to be the Thom space of $E$.

A commutative ring $T$-spectrum is a commutative monoid $(A,m,\unit)$ in $(\SH(k),\wedge,\SSp)$. The cohomology theory defined by a commutative $T$-spectrum is a ring cohomology theory satisfying a certain bigraded commutativity condition described by Morel.

We recall the essential properties of the cohomology theories represented by a commutative ring $T$-spectrum $A$.

(1) \textit{Localization:} for a closed embedding of varieties $i\colon Z\to X$ with a smooth $X$ and an open complement $j\colon U\to X$ put $z\colon X\to Th(i)=X/U$ for the canonical quotient map. Then we have a long exact sequence
$$
\xrightarrow{\partial} A^{*,*}(Th(i))\xrightarrow{z^A} A^{*,*}(X)\xrightarrow{j^A} A^{*,*}(U)\xrightarrow{\partial} A^{*+1,*}(Th(i))\xrightarrow{z^A}
$$
It is a special case of the cofiber long exact sequence.

(2) \textit{Nisnevich excision:} consider a Cartesian square of smooth varieties
$$
\xymatrix{
Z'\ar[r]^{i'} \ar[d]^{f'} & X' \ar[d]^{f} \\
Z\ar[r]^{i}  & X
}
$$
where $i$ is a closed embedding, $f$ is etale and $f'$ is an isomorphism. Then for the induced morphism $g\colon Th(i')\to Th(i)$ the corresponding morphism $g^A\colon A^{*,*}(Th(i))\to A^{*,*}(Th(i'))$ is an isomorphism. It follows from the fact that $g$ is an isomorphism in the homotopy category.

(3) \textit{Homotopy invariance:} for an $\A^n$-bundle $p\colon E\to X$ over a variety $X$ the induced homomorphism $p^A\colon A^{*,*}(X)\to A^{*,*}(E)$ is an isomorphism.

(4) \textit{Mayer-Vietoris:} if $X=U_1\cup U_2$ is a union of two open subsets $U_1$ and $U_2$ then there is a natural long exact sequence
$$
\to A^{*,*}(X)\to A^{*,*}(U_1)\oplus A^{*,*}(U_2)\to A^{*,*}(U_1\cap U_2)\to A^{*+1,*}(X)\to
$$

(5) \textit{Cup-product:} for a motivic space $Y$ we have a functorial graded ring structure
$$
\cup\colon A^{*,*}(Y)\times A^{*,*}(Y)\to A^{*,*}(Y).
$$
Moreover, let $i_1\colon Z_1\to X$ and $i_2\colon Z_2\to X$ be closed embeddings and put $i_{12}\colon Z_1\cap Z_2\to X$. Then we have a functorial, bilinear and associative cup-product
$$
\cup \colon A^{*,*}(Th(i_1))\times A^{*,*}(Th(i_2)) \to A^{*,*}(Th(i_{12})).
$$
In particular, setting $Z_1=X$ we obtain an $A^{*,*}(X)$-module structure on $A^{*,*}(Th(i_2))$. All the morphisms in the localization sequence are homomorphisms of $A^{*,*}(X)$-modules.

We will sometimes omit $\cup$ from the notation.

(6) \textit{Module structure over stable cohomotopy groups:} for every motivic space $Y$ we have a homomorphism of graded rings $\pi^{*,*}(Y)\to A^{*,*}(Y)$, which defines a $\pi^{*,*}(pt)$-module structure on $A^{*,*}(Y)$. For a smooth variety $X$ the ring $A^{*,*}(X)$ is a graded $\pi^{*,*}(pt)$-algebra via $\pi^{*,*}(pt)\to \pi^{*,*}(X)\to A^{*,*}(X)$.

(7) \textit{Graded $\epsilon$-commutativity} \cite{Mor1}: let $\epsilon\in \pi^{0,0}(pt)$ be the element corresponding under the suspension isomorphism to the morphism $T\to T, x\mapsto -x$. Then for every motivic space $X$ and $a\in A^{i,j}(X), b\in A^{p,q}(X)$ we have
$$
ab =(-1)^{ip}\epsilon^{jq}ba.
$$
Recall that $\epsilon^2=1$.

\section{Special linear orientation.}
In this section we recall the notion of a special linear orientation introduced in \cite{PW3} and establish some of its basic properties.

\begin{definition}
A \textit{special linear bundle} over a variety $X$ is a pair $(E, \lambda)$ with $E\to X$ a vector bundle and $\lambda \colon \det E\xrightarrow{\simeq}\triv_X$ an isomorphism of line bundles. An isomorphism $\phi\colon (E, \lambda)\xrightarrow{\simeq}(E' , \lambda' )$ of special linear vector bundles is an isomorphism $\phi\colon E\xrightarrow{\simeq} E'$ of vector bundles such that $\lambda' \circ (\det \phi)= \lambda$. For a special linear bundle $\Tc=(E,\lambda)$ we usually denote by the same letter $\Tc$ the total space of the bundle $E$.
\end{definition}

\begin{definition}
Consider a trivialized $\rank n$ bundle $\triv_X^n$ over a smooth variety $X$. There is a canonical trivialization $\det \triv^n_X\xrightarrow{\simeq} \triv_X$. We denote the corresponding special linear bundle by $(\triv_X^n, 1)$ and refer to it as the \textit{trivialized special linear bundle}.
\end{definition}

\begin{lemma}
\label{lem_triv_push}
Let $(E,\lambda)$ be a special linear bundle over a smooth variety $X$ such that $E\cong \triv_X^n$. Then there exists an isomorphism of special linear bundles
$$
\phi \colon (E,\lambda)\xrightarrow{\simeq} (\triv_X^n,1).
$$
\end{lemma}
\begin{proof}
An exact sequence of algebraic groups
$$
1\to SL_n \to GL_n \xrightarrow{\det} \Gm \to 1
$$
provides an exact sequence of pointed sets
$$
H^0(X,GL_n)\xrightarrow{p} H^0(X,\Gm)\to H^1(X,\SL_n) \xrightarrow{i} H^1(X,GL_n)
$$
There is a splitting $\Gm \to GL_n$ for $\det$, so $p$ is surjective. Hence we have $\ker i =\{*\}$ and this means that, up to an isomorphism of special linear bundles, there exists only one trivialization $\lambda\colon \det\triv^n_X\to \triv_X$.
\end{proof}

\begin{lemma}
Let $E_1$ be a subbundle of a vector bundle $E$ over a smooth variety $X$. Then there are canonical isomorphisms
\begin{enumerate}
\item
$\det E_1 \otimes \det (E/E_1) \cong \det E$,
\item
$\det E^\vee \cong (\det E)^\vee$.
\end{enumerate}
\label{quottriv}
\end{lemma}
\begin{proof}
These isomorphisms are induced by the corresponding vector space isomorphisms. In the first case we have
$\Lambda^m V_1 \otimes \Lambda^n (V/V_1) \xrightarrow{\simeq} \Lambda^{m+n}V$ with
$$
v_1\wedge {\dots}\wedge v_m \otimes \overline{w}_1\wedge{\dots}\wedge \overline{w}_n \mapsto v_1\wedge {\dots}\wedge v_m \wedge w_1\wedge {\dots}\wedge w_n .
$$
For the second isomorphism consider the perfect pairing
$$
\phi\colon\Lambda^n V\times \Lambda^n V^\vee \to k
$$
defined by
\[
\phi(v_1\wedge {\dots}\wedge v_n,f_1\wedge {\dots}\wedge f_n)=\sum_{\sigma\in S_n}sign(\sigma)f_{\sigma(1)}(v_1)\cdot{\dots} \cdot f_{\sigma(n)}(v_n).\qedhere
\]
\end{proof}

\begin{definition}
Let $\Tc=(E,\lambda_E)$ be a special linear bundle over a smooth variety $X$ and let $\Tc'=(E',\lambda_{E'})$ with $E'\le E$ be a special linear subbundle. By {Lemma \ref{quottriv}} we have canonical trivializations $\lambda_{E^\vee}\colon \det E^\vee \xrightarrow{\simeq} \triv_X$ and $\lambda_{E/E'}\colon \det (E/E')\xrightarrow{\simeq} \triv_X$. The special linear bundle $\Tc^\vee=(E^\vee,\lambda_{E^\vee})$ is called the \textit{dual special linear bundle} and the special linear bundle $\Tc/\Tc'=(E/E',\lambda_{E/E'})$ is called the \textit{quotient special linear bundle}. For a pair $\Tc_1=(E_1,\lambda_{E_1}), \Tc_2=(E_2,\lambda_{E_2})$ of special linear bundles over a smooth variety $X$ we put $\Tc_1\oplus\Tc_2=(E_1\oplus E_2,\lambda_{E_1}\otimes\lambda_{E_2})$ and refer to it as the \textit{direct sum of special linear bundles}.
\end{definition}

\begin{definition}
Let $A^{*,*}(-)$ be a ring cohomology theory represented by a $T$-spectrum $A$. A \textit{(normalized) special linear orientation} on $A^{*,*}(-)$ is a rule which assigns to every special linear bundle $\Tc$ of rank $n$ over a smooth variety $X$ a class $th(\Tc)\in A^{2n,n}(Th(\Tc))$ satisfying the following conditions \cite[Definition~5.1]{PW3}:
\begin{enumerate}
\item
For an isomorphism $f\colon \Tc\xrightarrow{\simeq} \Tc'$ we have $th(\Tc)=f^Ath(\Tc')$.
\item
For a morphism $r\colon Y\to X$ we have $r^Ath(\Tc)=th(r^*\Tc)$. 
\item
The maps $-\cup th(\Tc)\colon A^{*,*}(X)\to A^{*+2n,*+n}(Th(\Tc))$ are isomorphisms.
\item
We have
$$
th(\Tc_1\oplus \Tc_2)=q_1^Ath(\Tc_1)\cup q_2^Ath(\Tc_2),
$$
where $q_1,q_2$ are projections from $\Tc_1\oplus \Tc_2$ onto its summands. Moreover, for the zero bundle $\mathbf{0}\to pt$ we have $th(\mathbf{0})=1\in A^{0,0}(pt)$.
\item (normalization)
For the trivialized line bundle over a point we have $th(\triv_{pt},1)=\Sigma_T 1\in A^{2,1}(T)$.
\end{enumerate}
The isomorphism $-\cup th(\Tc)$ is the \textit{Thom isomorphism}.
The class $th(\Tc)$ is the \textit{Thom class} of the special linear bundle, and
$$
e(\Tc)=z^Ath(\Tc)\in A^{2n,n}(X)
$$
with natural $z\colon X\to Th(\Tc)$ is its \textit{Euler class}. A ring cohomology theory with a fixed (normalized) special linear orientation is called an {\it $SL$-oriented cohomology theory}.
\end{definition}

\begin{lemma}
\label{lem_eulermult}
Let $A^{*,*}(-)$ be an $SL$-oriented cohomology theory, let $\Tc$ be a special linear bundle over a smooth variety $X$ and let $\Tc_1\le \Tc$ be a special linear subbundle. Then $e(\Tc)=e(\Tc_1)e(\Tc/\Tc_1)$.
\end{lemma}
\begin{proof}
There is an $\mathbb{A}^r$-bundle $p\colon Y\to X$ such that 
\[
p^*\Tc\cong p^*\Tc_1\oplus p^*(\Tc/\Tc_1),
\]
so the claim follows from the homotopy invariance and the multiplicativity of the Euler class with respect to the direct sums. The variety $Y$ could be constructed in the following way. The fiber over a point $x\in X$ consists of the vector subspaces $E|_x\le \Tc|_x$ such that $\Tc|_x = \Tc_1|_x\oplus E|_x$. This construction could be performed locally and then glued into the variety $Y$.
\end{proof}

\begin{rem}
For a rank $2n$ special linear bundle $\Tc$ over a variety $X$ we have $th(\Tc)\in A^{4n,2n}(Th(\Tc))$ and $e(\Tc)\in A^{4n,2n}(X)$, so these classes are universally central.
\end{rem}

Recall that a symplectic bundle is a special linear bundle in a natural way, so having a special linear orientation we have the Thom classes also for symplectic bundles, thus an $SL$-oriented cohomology theory is also symplectically oriented. We recall the definition of the Borel classes theory (cf. \cite[Definition~14.1]{PW1}) that could be developed for a symplectically oriented cohomology theory. Note that our terminology is slightly different from the one used in loc. cit.: we refer to the ``Pontryagin classes'' in the sense of loc. cit. as ``Borel classes''.

\begin{definition}
Let $A^{*,*}(-)$ be a cohomology theory represented by a $T$-spectrum $A$. A \textit{Borel classes theory} on $A^{*,*}(-)$ is a rule which assigns to every symplectic bundle $(E,\phi)$ over every smooth variety $X$ a system of \textit{Borel classes} $b_i(E,\phi)\in A^{4i,2i}(X)$ for all $i\ge 1$ satisfying
\begin{enumerate}
\item
For $(E_1,\phi_1)\cong (E_2,\phi_2)$ we have $b_i(E_1,\phi_1)=b_i(E_2,\phi_2)$ for all $i$.
\item
For a morphism $r\colon Y\to X$ and a symplectic bundle $(E,\phi)$ over $X$ we have $r^A(b_i(E,\phi))=b_i(r^*(E,\phi))$ for all $i$.
\item
For the tautological rank $2$ symplectic bundle $(E,\phi)$ over $$HP^1=Sp_4/(Sp_2\times Sp_2)$$ the elements $1,b_1(E,\phi)$ form an $A^{*,*}(pt)$-basis of $A^{*,*}(HP^1)$.
\item
For a rank $2$ symplectic bundle $(V,\phi)$ over $pt$ we have $b_1(V,\phi)=0$.
\item
For an orthogonal direct sum of symplectic bundles $(E,\phi)\cong (E_1,\phi_1)\perp (E_2,\phi_2)$ we have
$$
b_i(E,\phi)=b_i(E_1,\phi_1)+\sum_{j=1}^{i-1} b_{i-j}(E_1,\phi_1)b_j(E_2,\phi_2)+b_i(E_2,\phi_2)
$$
for all $i$.
\item
For $(E,\phi)$ of rank $2r$ we have $b_i(E,\phi)=0$ for $i>r$.
\end{enumerate}
We set $b_*(E,\phi)=1+\sum_{i=1}^{\infty}b_i(E,\phi)t^i$ to be the \textit{total Borel class}.
\end{definition}

\begin{rem}
For a sympletically oriented cohomology theory $A^{*,*}(-)$ the canonical Borel classes theory could be defined in the following way. For a symplectic vector bundle $(E,\phi)$ of rank $2n$ over a smooth variety $X$ one puts 
\[
b_n(E,\phi)=z^Ath(E,\phi)\in A^{4n,2n}(X)
\]
for the natural map $z\colon X\to Th(E)$. Then one may define the lower Borel classes using the symplectic version of the projective bundle theorem, see \cite{PW1} for the details. Note that since these Borel classes are similar to the symplectic Borel classes in topology and not to the Pontryagin classes, we omit the sign in the above formula for the top Borel class.
\end{rem}

Every oriented cohomology theory possesses a special linear orientation via $th(E,\lambda)=th(E)$, so one can consider $K$-theory or algebraic cobordism represented by $MGL$ as examples. We have two main instances of the theories with a special linear orientation but without a general one. The first one is hermitian $K$-theory \cite{Sch} represented by the spectrum $\BO$ \cite{PW2}. The special linear orientation of $\BO^{*,*}$ via Koszul complexes could be found in \cite{PW2}. The second one is universal in the sense of \cite[Theorem~5.9]{PW3} and represented by the algebraic special linear cobordism spectrum $MSL$ \cite[Definition~4.2]{PW3}.

\begin{definition}
From now on $A^{*,*}(-)$ is an $SL$-oriented ring cohomology theory represented by a commutative monoid in $\SH(k)$.
\end{definition}

\begin{lemma}
For a smooth variety $X$ we have 
\[
th(\triv_X^n,1)=\Sigma_T^{n} 1,\quad th(\triv _X,-1)=\Sigma_T(\epsilon).
\]
\end{lemma}
\begin{proof}
The first equality follows from the conditions (4) and (5). For the second equality consider the morphism $f^A\colon Th(\triv_X)\to Th(\triv_X)$ corresponding to the isomorphism of the special linear bundles $(\triv_X,-1)\to (\triv_X,1)$, $v\mapsto -v$. By the very definition we have $f^A(\Sigma_T(1))=\Sigma_T(\epsilon)=\Sigma_T(1)\cup \epsilon$. On the other hand, functoriality of Thom classes together with normalization yields 
\[
th(\triv_X,-1)=f^A(th(\triv_X,1))=f^A(\Sigma_T(1))=\Sigma_T(1)\cup \epsilon=\Sigma_T(\epsilon).\qedhere
\]
\end{proof}

\begin{lemma}
\label{lemm_epsilon}
Let $(E,\lambda_E)$ be a special linear bundle over a smooth variety $X$. Then
$$
e(E,\lambda_E)=\epsilon \cup e(E,-\lambda_E).
$$
\end{lemma}
\begin{proof}
Consider the bundle $E\oplus \triv_X$ and denote the projections onto the summands by $q_1,q_2$. We have
$$
(E\oplus \triv_X,\lambda_E \otimes 1)=(E\oplus \triv_X,(-\lambda_E) \otimes -1),
$$
hence
$$
q_1^*th(E,\lambda_E)\cup q_2^*\Sigma_T 1=q_1^*th(E,-\lambda_E)\cup q_2^*\Sigma_T \epsilon.
$$
By the suspension isomorphism we obtain
$$
th(E,\lambda_E)=th(E,-\lambda_E)\cup \epsilon,
$$
hence $e(E,\lambda_E)=\epsilon \cup e(E,-\lambda_E)$.
\end{proof}

\begin{lemma}
Let $\Tc$ be a rank $2$ special linear bundle over a smooth variety $X$. Then
$\Tc\cong \Tc^\vee$ and $e(\Tc)=e(\Tc^\vee).$
\end{lemma}
\begin{proof}
Set $\Tc=(E,\lambda_E)$. The trivialization $\lambda_E\colon \Lambda^2 E\xrightarrow{\simeq} \triv_X$ defines a symplectic form on $E$ and an isomorphism $\phi\colon E\xrightarrow{\simeq}E^\vee$, thus it is sufficient to check that
$$
\lambda_{E^\vee}\circ \det \phi=\lambda_E.
$$
It could be checked locally, so we can suppose that $E\cong \triv^2_X$ and, in view of Lemma~\ref{lem_triv_push}, $\left(E,\lambda_E\right)\cong \left(\triv^2_X,1\right)$. Fixing a basis $\{e_1,e_2\}$ such that $e_1\wedge e_2=1$ and taking the dual basis $\{e_1^\vee,e_2^\vee\}$ for $(\triv^{2}_X)^{\vee}$ we have
$$\phi(e_1)=(e_1\wedge -)=e_2^\vee,\qquad \phi(e_2)=(e_2\wedge -)=-e_1^\vee.$$
Thus we obtain
$$
\det \phi (e_1\wedge e_2)=e_2^\vee\wedge(-e_1^\vee)=e_1^\vee\wedge e_2^\vee
$$
and
\[
\lambda_{E^\vee}\det \phi (e_1\wedge e_2)=\lambda_{E^\vee}(e_1^\vee\wedge e_2^\vee)=1. \qedhere
\]
\end{proof}

\begin{definition}
For a vector bundle $E$ we denote by $E^0$ the complement to the zero section. For a special linear bundle $\Tc=(E,\lambda)$ we put $\Tc^0=E^0$.
\end{definition}

\begin{definition}
\label{def_Gys1}
Let $\Tc$ be a rank $n$ special linear bundle over a smooth variety $X$. The \textit{Gysin sequence} is a long exact sequence
$$
\xrightarrow{\partial} A^{*-2n,*-n}(X)\xrightarrow{\cup e(\Tc)} A^{*,*}(X)\to A^{*,*}(\Tc^0)\xrightarrow{\partial} A^{*-2n+1,*-n}(X)\to
$$
obtained from the localization sequence for the zero section $X\to \Tc$ via the homotopy invariance and the Thom isomorphism.

\end{definition}

\begin{lemma}
\label{lemm_dual_eu}
Let $(E,\lambda_E)$ be a special linear bundle over a smooth variety $X$.
\begin{enumerate}
\item
Let $\lambda_E '$ be any other trivialization of $\det E$. Then one has
$$
A^{0,0}(X)\cup e(E,\lambda_E)=A^{0,0}(X)\cup e(E,\lambda_E').
$$
\item
For the dual special linear bundle $(E^\vee,\lambda_{E^\vee})$ one has
$$
A^{0,0}(X)\cup e(E,\lambda_E)=A^{0,0}(X)\cup e(E^\vee,\lambda_{E^\vee}).
$$
\end{enumerate}
\end{lemma}
\begin{proof}
Set $n=\rank E$ and denote the projections $E^0\to X$ and $E^{\vee 0}\to X$ by $p$ and $p'$ respectively.
\begin{enumerate}
\item
Consider the Gysin sequences corresponding to the trivializations $\lambda_E$ and $\lambda_E'$.
$$
\xymatrix{
{}\ar[r] &  A^{0,0}(X)\ar[rr]^(0.50){\cup e(E,\lambda_E)} & & A^{2n,n}(X) \ar[r]^{p^A}\ar[d]^{=}& A^{2n,n}(E^0) \ar[d]^{=} \ar[r] & {} \\
{}\ar[r] &  A^{0,0}(X)\ar[rr]^(0.50){\cup e(E,\lambda_E')} & & A^{2n,n}(X) \ar[r]^{p^A}& A^{2n,n}(E^0) \ar[r] & {}
}
$$
We have
$$
A^{0,0}(X)\cup e(E,\lambda_E)=\ker p^A=A^{0,0}(X)\cup e(E,\lambda_{E}').
$$
\item
Consider
$$
Y=\left\{\left.(v,f)\in E\times_X E^\vee\,\right|\, f(v)=1\right\}.
$$
Projections $p_1\colon Y\to E^0$ and $p_2\colon Y\to E^{\vee 0}$ have fibres isomorphic to $\A^{n-1}$, thus
$$
A^{*,*}(E^0)\cong A^{*,*}(Y)\cong A^{*,*}(E^{\vee 0})
$$
and we have a canonical isomorphism $ A^{*,*}(E^0)\cong A^{*,*}(E^{\vee 0})$ over $A^{*,*}(X)$.
Now proceed as in the first part and consider the Gysin sequences.
$$
\xymatrix{
{}\ar[r] &  A^{0,0}(X)\ar[rr]^(0.50){\cup e(E,\lambda_E)} & & A^{2n,n}(X) \ar[r]^{p^A}\ar[d]^{=}& A^{2n,n}(E^0) \ar[d]^{\cong} \ar[r] & {} \\
{}\ar[r] &  A^{0,0}(X)\ar[rr]^(0.50){\cup e(E^\vee,\lambda_{E^\vee})} & & A^{2n,n}(X) \ar[r]^{p'^A}& A^{2n,n}(E^{\vee0}) \ar[r] & {}
}
$$
We have
\[
A^{0,0}(X)\cup e(E,\lambda_E)=\ker p^A=\ker p'^A=A^{0,0}(X)\cup e(E^\vee,\lambda_{E^\vee}). \qedhere
\]

\end{enumerate}
\end{proof}

\begin{lemma}
\label{lemm_triv_section}
Let $\Tc$ be a special linear bundle over a smooth variety $X$ such that there exists a nowhere vanishing section $s\colon X\to \Tc$. Then $e(\Tc)=0$.
\end{lemma}
\begin{proof}
Set $\rank \Tc=n$ and consider the Gysin sequence
$$
{\dots}\to A^{0,0}(X)\xrightarrow{\cup e(\Tc)} A^{2n,n}(X) \xrightarrow{j^A} A^{2n,n}(\Tc^0)\to {\dots}
$$
The section $s$ induces a splitting $s^A$ for $j^A$, thus $j^A$ is injective and
\[
e(\Tc)=1\cup e(\Tc)=0.\qedhere
\]
\end{proof}

\section{Pushforwards along closed embeddings.}
In this section we give the construction of the pushforwards along the closed embeddings with special linear normal bundles for an $SL$-oriented cohomology theory. It is quite similar to the construction of such pushforwards for oriented \cite{PS,Ne1} or symplectically oriented \cite{PW1} cohomology theories and twisted Witt groups \cite{Ne2}, so we follow loc. cit. adapting it to the special linear context.

\begin{definition}
Let $i\colon Z\to X$ be a closed embedding of smooth varieties. The \textit{deformation space} $D(Z,X)$ is obtained as follows.
\begin{enumerate}
\item
Consider $X\times \A^1$.
\item
Blow-up it along $Z\times 0$.
\item
Remove the blow-up of $X\times 0$ along $Z\times 0$.
\end{enumerate}
This construction produces a smooth variety $D(Z,X)$ over $\A^1$. The fiber over $0$ is canonically isomorphic to $N_i$ while the fiber over $1$ is isomorphic to $X$ and we have the corresponding closed embeddings $i_0\colon N_i\to D(Z,X)$ and $i_1\colon X\to D(Z,X)$. There is a closed embedding $z\colon Z\times \A^1 \to D(Z,X)$ such that over $0$ it coincides with the zero section $s\colon Z\to N_i$ of the normal bundle and over $1$ it coincides with the closed embedding $i\colon Z\to X$. At last, we have a projection $p\colon D(Z,X) \to X$.

Thus we have homomorphisms of $A^{*,*}(X)$-modules (via $p^A$)
$$
\xymatrix{
A^{*,*}(Th(N_i)) & A^{*,*}(Th(z)) \ar[l]_(0.45){i_0^A} \ar[r]^{i_1^A} & A^{*,*}(Th(i)).
}
$$
These homomorphisms are isomorphisms since in the homotopy category $H_\bullet (k)$ we have isomorphisms $i_0\colon Th(N_i)\cong Th(z)$ and $i_1\colon Th(i)\cong Th(z)$ \cite[Theorem~2.23]{MV}. We set
$$
d^A_{i}=i_1^A\circ (i_0^A)^{-1}\colon A^{*,*}(Th(N_i))\to A^{*,*}(Th(i))
$$
to be the \textit{deformation to the normal bundle isomorphism}. The functoriality of the deformation space $D(Z,X)$ makes the deformation to the normal bundle isomorphism functorial.
\end{definition}

\begin{definition}
For a closed embedding $i\colon Z\to X$ of smooth varieties a \textit{special linear normal bundle} is a pair $(N_i,\lambda)$ with $N_i$ the normal bundle and $\lambda\colon \det N_i\xrightarrow{\simeq} \triv_Z$ an isomorphism of line bundles.
\end{definition}
\begin{definition}
Let $i\colon Z\to X$ be a closed embedding of smooth varieties with a rank $n$ special linear normal bundle $(N_i,\lambda)$. Denote by $\tilde{\imath}_A$ the composition of the Thom and deformation to the normal bundle isomorphisms,
\[
\tilde{\imath}_A=d^A_i\circ (-\cup th(N_i,\lambda))\colon A^{*,*}(Z)\xrightarrow{\simeq} A^{*+2n,*+n}(Th(i)).
\]
For the inclusion $z\colon X\to Th(i)$ the composition
\[
i_A=z^A\circ \tilde{\imath}_A\colon A^{*,*}(Z)\to A^{*+2n,*+n}(X)
\]
is the \textit{pushforward map}. Note that in general $i_A$ depends on the trivialization of $\det N_i$.
\end{definition}

\begin{rem}
We have an analogous definition of the pushforward map for a closed embedding $i\colon Z\to X$ in every cohomology theory possessing a Thom class for the normal bundle $N_i$. In particular, we have pushforwards in the stable cohomotopy groups for closed embeddings with a trivialized normal bundle $(N_i,\theta)$, where $\theta\colon N_i\xrightarrow{\simeq} \triv_Z^n$ is an isomorphism of vector bundles, since there is a Thom class $th(\triv_Z^n)=\Sigma^n_T 1$ and suspension isomorphism
$$
(-\cup \Sigma^{n}_T 1)\colon \pi^{*,*}(Z)\xrightarrow{\simeq}\pi^{*+2n,*+n}(Th(\triv_Z^n)).
$$
\end{rem}

\begin{definition}
Let $i\colon Z\to X$ be a closed embedding of smooth varieties with a rank $n$ special linear normal bundle. Then using the notation of pushforward maps the localization sequence boils down to
\[
\xrightarrow{\partial} A^{*-2n,*-n}(Z)\xrightarrow{i_A} A^{*,*}(X)\xrightarrow{j^A} A^{*,*}(X-Z)\xrightarrow{\partial} A^{*-2n+1,*-n}(Z)\xrightarrow{i_A} 
\]
We refer to this sequence as \textit{the Gysin sequence}, similar to Definition~\ref{def_Gys1}.
\end{definition}

In the rest of this section we sketch some properties of the pushforward maps. The next lemma is similar to \cite[Proposition~7.4]{PW1}.
\begin{lemma}
\label{lem_transv_push}
Consider the following cartesian diagram with all the involved varieties being smooth.
$$
\xymatrix{
X'=X\times_Y Y' \ar[rr]^(0.6){i'} \ar[d]^{g'}& & Y' \ar[d]^{g}\\
X \ar[rr]^{i}& & Y
}
$$
Let $i,i'$ be closed embeddings with special linear normal bundles $(N_i,\lambda)$ and let $(N_{i'},\lambda')\cong (g'^*N_i,g'^*\lambda)$. Then we have $g^A\tilde{\imath}_A=\tilde{\imath}'_Ag'^A$.
\end{lemma}
\begin{proof}
It follows from the functoriality of the deformation to the normal bundle and the functoriality of Thom classes.
\end{proof}

The next proposition is an analogue of \cite[Proposition~7.6]{PW1}.
\begin{proposition}
\label{prop_section}
Let $\Tc$ be a special linear bundle over a smooth variety $X$ with a section $s\colon X\to \Tc$ meeting the zero section $r$ transversally in $Y$. Then for the inclusion $i\colon Y\to X$ and every $b\in A^{*,*}(X)$ we have
$$
i_Ai^A(b)=b\cup e(\Tc).
$$
\end{proposition}
\begin{proof}
Let $z^A\colon A^{*,*}(Th(i))\to A^{*,*}(X)$ and $\overline{z}^A\colon A^{*,*}(Th(\Tc))\to A^{*,*}(\Tc)$ be the extension of supports maps and let $p \colon \Tc\to X$ be the structure map for the bundle. Consider the following diagram.
$$
\xymatrix{
A^{*,*}(X) \ar[rr]^{\tilde{r}_A}_{\cup th(\Tc)}\ar[d]^{i^A} & & A^{*,*}(Th(\Tc)) \ar[rr]^{\overline{z}^A}\ar[d]^{s^A} & & A^{*,*}(\Tc)\ar[d]^{s^A} \ar@<-0.5pc>[d]_{r^A}\\
A^{*,*}(Y) \ar[rr]^{\tilde{\imath}_A} & & A^{*,*}(Th(i)) \ar[rr]^{z^A} & & A^{*,*}(X) \ar@/_1.5pc/[u]_{p^A}
}
$$
The pullbacks along the two section of $p$ are inverses of the same isomorphism $p^A$, so $s^A=r^A$. The right-hand
square consists of pullbacks thus it is commutative.  The left-hand square commutes by Lemma~\ref{lem_transv_push}. Hence we have
\[
i_Ai^A(b)=z^A\tilde{\imath}_Ai^A(b)=r^A\overline{z}^A(b\cup th(\Tc))=b\cup e(\Tc). \qedhere
\]
\end{proof}

The pushforward maps are compatible with the compositions of the closed embeddings. The following proposition is similar to \cite[Proposition 5.1]{Ne2} and the same reasoning works out, so we omit the proof.
\begin{proposition}
\label{prop_composition}
Let $Z\xrightarrow{i}Y\xrightarrow{j}X$ be closed embeddings of smooth varieties with special linear normal bundles $(N_{ji},\lambda_{ji}),(N_i,\lambda_i),(i^*N_{ji}/N_i,\lambda_j)$ such that $\lambda_i\otimes\lambda_j=\lambda_{ji}$. Then
$$
j_Ai_A=(ji)_A.
$$
\end{proposition}


\section{Preliminary computations in the stable cohomotopy groups and the stable Hopf map.}

We are going to do preliminary computations involving $\pi^{*,*}$ and various motivic spheres. The main result of this section is Proposition~\ref{prop_stable_coh} proved by a rather lengthy computation. We track down all the involved canonical isomorphisms, so the formulas are a bit messy. 

Throughout this section we use $X=\A^{n+1}-\{0\}$ for a punctured affine space with $n\ge 1$. Let $x=(1,1,0,\dots,0)$ be a point on $X$. First of all recall the following well-known isomorphisms \cite[Lemma 2.15, Example 2.20]{MV}.


\begin{definition}
Set $\sigma=\sigma_2^{-1}\sigma_1\colon (X,x)\xrightarrow{\simeq} (\Gm,1)\wedge T^{\wedge n} $ for the canonical isomorphism in the homotopy category. It is defined via
\[
(X,x)\xrightarrow{\sigma_1} X/((\A^1\times (\A^{n}-\{0\}))\cup (\{1\}\times \A^n)) \xleftarrow{\sigma_2} (\Gm,1)\wedge T^{\wedge n},
\]
where $\sigma_1$ is induced by the identity map on $X$ and $\sigma_2$ is induced by the natural embedding $\Gm\times \A^n\subset X$. Recall that $\sigma_1$ is an isomorphism since $(\A^1\times (\A^{n}-\{0\}))\cup (\{1\}\times \A^n)$ is $\A^1$-contractible, while $\sigma_2$ is induced by the excision isomorphism $(\Gm_+,+)\wedge T^{\wedge n} \cong X/(X - (\A^1\times \{(0,0,\ldots,0)\})),$ so it is an isomorphism as well.

We write $s=s_2^{-1}s_1\colon (\A^{2}-\{0\},(1,1))\xrightarrow{\simeq} (\Gm,1)\wedge T$ for this isomorphism in the particular case of $n=1$.
\end{definition}

Another isomorphism that we need could be easily expressed via the cone construction.

\begin{definition}
Let $i\colon Y\to Z$ be a morphism of pointed motivic spaces. The space $Cone(i)$ defined via the cocartesian square
\[
\xymatrix{
Y \ar[r]^{i}\ar[d]^{in_1} & Z \ar[d] \\
Y\wedge \Delta^1 \ar[r] & Cone(i)
}
\]
is called the \textit{cone of the morphism $i$}.
\end{definition}

\begin{definition}
\label{def_rho}
Set $\rho=\rho_2\circ\rho_1^{-1}\colon T\xrightarrow{\simeq} (\Gm,1)\wedge S^1_s=S^{2,1}$ for the canonical isomorphism in the homotopy category defined via
\[
T\xleftarrow{\rho_1} Cone(i_\rho) \xrightarrow{\rho_2} (\Gm,1)\wedge S^1_s
\]
where $i_\rho$ stands for the natural embedding $(\Gm,1)\to (\A^1,1)$ and the isomorphisms $\rho_1$ and $\rho_2$ are induced by the maps $\Delta^1\to pt$ and $\A^1\to pt$ respectively.
\end{definition}

\begin{definition}
For every pointed motivic space $Y$ put 
\[
\Sigma_T=(id_{Y}\wedge\rho)^{\pi}\Sigma^{2,1}\colon \pi^{*,*}(Y)\to \pi^{*+2,*+1}(Y\wedge T)
\]
and set $\Sigma_T^{n}=\Sigma_T \circ \Sigma_T \circ \ldots \circ \Sigma_T$ for the $n$-fold composition.
\end{definition}

Consider the localization sequence for the embedding $\{0\}\to \A^{n+1}$,
$$
\to\pi^{2n,n}(T^{\wedge n+1}) \to \pi^{2n,n}(\A^{n+1}) \to \pi^{2n,n}(X)\xrightarrow{\partial} \pi^{2n+1,n}(T^{\wedge n+1})\to
$$
Canonical isomorphisms described above together with the choice of the point $x$ on $X$ provide a splitting for the connecting homomorphism $\partial$. We discuss it in the next lemma. Put 
\[
\tau\colon T^{\wedge n}\wedge (\Gm,1)\to  (\Gm,1)\wedge T^{\wedge n},\quad \tau_c\colon  T^{\wedge n}\wedge T\to T\wedge T^{\wedge n} 
\]
 for the twisting isomorphisms defined via $(x_0,x_1,\ldots,x_n)\mapsto (x_n,x_0,\ldots,x_{n-1})$.

\begin{lemma}
\label{lem_cohsplit}
For the canonical morphism $r\colon (X_+,+)\to (X,x)$ we have 
\[
\partial r^\pi=(\tau_c^\pi)^{-1}(id\wedge\rho)^\pi\Sigma^{1,0}\tau^\pi(\sigma^\pi)^{-1}.
\]
\end{lemma}
\begin{proof}
On the right-hand side we have 
\begin{multline*}
(\tau_c^\pi)^{-1}(id\wedge\rho)^\pi\Sigma^{1,0}\tau^\pi(\sigma^\pi)^{-1}=(\tau_c^\pi)^{-1}(id\wedge\rho)^\pi(\sigma^{-1}\tau\wedge id)^\pi\Sigma^{1,0}=\\
=((\sigma^{-1}\tau\wedge id)(id\wedge \rho)\tau_c^{-1})^\pi\Sigma^{1,0}.
\end{multline*}
Put $Y=(\A^1\times (\A^{n}-\{0\}))\cup (\{1\}\times \A^n)$. Let $i_Y\colon X/Y\to \A^{n+1}/Y$, $i_G\colon T^{\wedge n}\wedge (\Gm,1) \to T^{\wedge n}\wedge (\A^1,1)$ and $i_+\colon (X_+,+)\to (\A^{n+1}_+,+)$ be the natural embeddings and let $j_1\colon (\A^{n+1}_+,+)\to Cone(i_+)$ and $j_2\colon Cone(i_+)\to Cone(j_1)$ be the canonical maps for the cone construction. 

Consider the following diagram.
\[
\xymatrix{
 T^{\wedge n}\wedge (\Gm,1)\wedge S^1_s  \ar[d]_{\tau\wedge id}^{\simeq} & T^{\wedge n}\wedge Cone(i_\rho) \ar[l]_(0.45){id\wedge \rho_2}^(0.45){\simeq} \ar[r]^(0.6){id\wedge \rho_1}_(0.6){\simeq} \ar[d]_{w}^{\simeq} & T^{\wedge (n+1)} \ar[r]^{\tau_c}_{\simeq}  & T^{\wedge (n+1)}   \\
 (\Gm,1)\wedge T^{\wedge n}\wedge S^1_s  \ar[d]_{\sigma_2\wedge id}^{\simeq}  & Cone(i_G) \ar[d]_{t}^{\simeq} & &\\
(X/Y)\wedge S^1_s & Cone(i_Y) \ar[uurr]^{\psi_3}_{\simeq} \ar[l]_{u}^{\simeq}&  & Cone(i_+) \ar[d]^{j_2} \ar[uu]^{\simeq}_{\psi_1} \ar[ll]_{v}^{\simeq} \\
(X,x)\wedge S^{1}_s \ar[u]^{\sigma_1\wedge id}_{\simeq} & &(X_+,+)\wedge S^1_s  \ar[ll]_{r\wedge id} &  Cone(j_1)  \ar[l]_(0.4){\psi_2}^(0.4){\simeq} 
}
\]
Here 
$\psi_1,\psi_2$ and $\psi_3$ are induced by $\Delta^1\to pt$, $v$ is induced by $\A^{n+1}\to \A^{n+1}/Y$ and $X_+\to X/Y$, $u$ is induced by $\A^{n+1}/Y\to pt$, $w$ is an obvious isomorphism $T^{\wedge n}\wedge Cone(i_\rho)\cong Cone(id\wedge i_\rho)$ and $t$ is induced by the commutative square
\[
\xymatrix{
T^{\wedge n}\wedge (\Gm, 1) \ar[r] \ar[d]^{i_G} & X/Y \ar[d]^{i_Y} \\
T^{\wedge n}\wedge (\A^1,1) \ar[r]^(0.57){\tau_c'} & \A^{n+1}/Y,
}
\]
where $\tau_c'(x_0,x_1,\ldots,x_n)=(x_n,x_0,\ldots,x_{n-1})$.

 One can easily verify that the large diagram is commutative. By the very definition we have 
\[
\partial r^\pi=(\psi_2j_2\psi_1^{-1})^\pi\Sigma^{1,0}r^\pi=((r\wedge id)\psi_2j_2\psi_1^{-1})^\pi\Sigma^{1,0},
\]
thus it is sufficient to show
\[
(r\wedge id)\psi_2j_2\psi_1^{-1}
=(\sigma^{-1}\tau\wedge id)(id\wedge \rho)\tau_c^{-1}
\]
and it follows from the commutativity of the above diagram.
%
\end{proof}

\begin{definition}
The \textit{Hopf map} is the morphism of pointed motivic spaces
$$
H\colon (\A^2-\{0\},(1,1)) \to (\Proj^1,[1:1])
$$
given by $H(x,y)=[x,y]$. Let $\vartheta=\vartheta_2^{-1}\vartheta_1$ be the composition 
\[
\vartheta\colon (\PP^1,[1:1])\xrightarrow{\vartheta_1}\PP^1/\A^1\xleftarrow{\vartheta_2} T,
\]
where $\vartheta_1$ is induced by the identity map on $\PP^1$ and $\vartheta_2$ is the excision isomorphism given by $\vartheta_2(x)=[x:1]$. Then the \textit{stable Hopf map} is the unique element $\eta \in \pi^{-1,-1}(pt)$ such that $s^{\pi}\Sigma_T\Sigma^{1,1}\eta=\Sigma^\infty_T (\rho \vartheta H)$, i.e. $\eta$ is the stabilization of $H$ moved to $\pi^{-1,-1}(pt)$ via the canonical isomorphisms.
\end{definition}

\begin{lemma}
\label{lem_tildeH}
Let $\widetilde{H}\colon (\A^2-\{0\},(1,1)) \to (\Proj^1,[1:1])$ be the morphism of pointed motivic spaces given by $\widetilde{H}(x,y)=[y:x]$ and let
$\widetilde{\eta}\in \pi^{-1,-1}(pt)$ be the unique element such that $s^{\pi}\Sigma_T\Sigma^{1,1}\widetilde{\eta}=\Sigma^\infty_T (\rho \vartheta\widetilde{H})$. Then $\widetilde{\eta}=\epsilon \cup\eta$.
\end{lemma}
\begin{proof}
Let $\phi\colon (\A^2-\{0\},(1,1))\to (\A^2-\{0\},(1,1))$ be the reflection given by $\phi(x,y)=(y,x)$. Put $Y=(\A^1\times (\A^2-\{0\}))\cup (\{1\}\times \A^2)$ and consider the following commutative diagram.
\[
\xymatrix{
(\A^2-\{0\},(1,1))\wedge T \ar[r]^{\phi\wedge -id_T} \ar[d]_{\psi_2}^{\simeq}& (\A^2-\{0\},(1,1))\wedge T \ar[d]_{\psi_2}^{\simeq}\\
(\A^3-\{0\})/Y \ar[r]^{\psi_3}_{\simeq} &  (\A^3-\{0\})/Y \\
(\Gm,1)\wedge T \wedge T \ar[r]^{id\wedge (-id_{T\wedge T})} \ar[u]^{\psi_1}_{\simeq} & (\Gm,1)\wedge T \wedge T \ar[u]^{\psi_1}_{\simeq}
}
\]
Here $\psi_1$ is induced by the inclusion $\Gm\times\A^2\to \A^3-\{0\}$, $\psi_2$ is given by $\psi_2(x,y,z)=(\frac{x+y}{2},\frac{x-y}{2},z)$ and $\psi_3(x,y,z)=(x,-y,-z)$. All the morphisms $\psi_i$ are isomorphisms: $\psi_1$ is an excision isomorphism, $\psi_3$ is an involution and $\psi_2$ could be decomposed in an obvious way
\[
(\A^2-\{0\},(1,1))\wedge T \xrightarrow{\psi_2'} (\A^2-\{0\})/((\A^1\times \Gm)\cup (\{1\}\times \A^1))\wedge T \to (\A^3-\{0\})/Y
\]
with the first map $\psi_2'(x,y,z)=(\frac{x+y}{2},\frac{x-y}{2},z)$ being an isomorphism since $((\A^1\times \Gm)\cup (\{1\}\times \A^1))$ is $\A^1$-contractible and the second map being an excision isomorphism. It is well-known that $-id_{T\wedge T}=id_{T\wedge T}$ in the homotopy category, so we obtain $\phi\wedge -id_T=\psi_2^{-1}\psi_1 (id\wedge id_{T\wedge T})\psi_1^{-1}\psi_2=id\wedge id_T$ yielding 
\[
(\rho \vartheta\widetilde{H})\wedge id_T=((\rho \vartheta\widetilde{H})\wedge id_T)(\phi\wedge -id_T)=(\rho \vartheta H)\wedge -id_T.
\]
Taking the $\Sigma^{\infty}_T$-suspension and using the suspension isomorphism $\Sigma_T^{-1}$ we get
\[
\Sigma^\infty_T (\rho \vartheta H)\cup\epsilon=\Sigma^\infty_T (\rho \vartheta\widetilde{H}).
\]
The suspension isomorphisms as well as $\rho^{\pi}$ and $s^{\pi}$ are homomorphism of $\pi^{0,0}(pt)$-modules, and $\epsilon$ is central, thus
\[
s^{\pi}\Sigma_T\Sigma^{1,1}(\epsilon\cup\eta)=\Sigma^\infty_T (\rho \vartheta H)\cup\epsilon=\Sigma^\infty_T (\rho \vartheta\widetilde{H})=s^{\pi}\Sigma_T\Sigma^{1,1}(\widetilde{\eta}).
\]
The claim follows via taking $(s^\pi)^{-1}$ and desuspending.
\end{proof}

Recall that for the stable cohomotopy groups we have canonical Thom classes for the trivialized vector bundles $th(\triv_X^n)=\Sigma^{n}_T1$ and pushforwards $i_\pi$ for the closed embeddings with a trivialized normal bundle $(N_i,\theta)$.

We fix the following notation. Let $i\colon \Gm \to X$ be a closed embedding to the zeroth coordinate given by $i(t)=(t,0,\dots,0)$. Identify the normal bundle
$$
N_i\cong U=\Gm\times \A^{n}\subset X
$$
with the Zariski neighbourhood $U$ of $\Gm$ and define the trivialization $\theta\colon U \xrightarrow{\simeq} \triv_\Gm^{n}$ via
$$
\theta(t,x_1,\dots,x_{n})=(t,x_1/t,x_2,\dots,x_{n}).
$$
There is a pushforward map
$$
i_\pi\colon \pi^{0,0}(\Gm)\to \pi^{2n,n}(X)
$$
induced by the trivialization $\theta$.

%

\begin{proposition}
\label{prop_stable_coh}
In the above notation we have $\partial i_\pi (1)=(-1)^n \epsilon\cup \Sigma^{n+1}_T\eta$.
\end{proposition}
\begin{proof}
From the construction of the pushforward map we have
$$
i_\pi(1)=z^\pi d_i^\pi (th(U,\theta))
$$
with $z^\pi\colon \pi^{*,*}(Th(i))\to \pi^{*,*}(X)$ being a support extension and $d_i^\pi$ a deformation to the normal bundle isomorphism. Represent $i$ as a composition
\[
i\colon \Gm\xrightarrow{i_1} U\xrightarrow{i_2} X
\]
and let $w\colon Th(i_1)\xrightarrow{\simeq} Th(i)$ be the induced isomorphism in the homotopy category. Recall that for the total space of the vector bundle $U$ there is a natural isomorphism \cite[proof of Proposition~3.1]{Ne2} $D(\Gm,U)\cong U\times \A^1$ and $d_{i_1}^\pi=id$. By the functoriality of the deformation construction we have $d_i^\pi=(w^\pi)^{-1}$, so we need to compute
$$
\partial z^\pi (w^{\pi})^{-1}(th(U,\theta)).
$$
Decomposing $z$ in
\[
z\colon (X_+,+)\xrightarrow{r} (X,x)\xrightarrow{z_1} Th(i)
\]
and using Lemma~\ref{lem_cohsplit} we obtain
\begin{multline*}
\partial z^\pi (w^{\pi})^{-1}(th(U,\theta))=\partial r^\pi z_1^\pi (w^{\pi})^{-1}(th(U,\theta))=\\
=(\tau_c^\pi)^{-1}(id\wedge\rho)^\pi\Sigma^{1,0}(\tau^\pi(\sigma^\pi)^{-1}z_1^\pi (w^{\pi})^{-1}(th(U,\theta))).
\end{multline*}
We can represent the Thom class $th(U,\theta)\in \pi^{2n,n}(Th(i_1))$ by $\Sigma_T^\infty$-suspension of the composition 
\[
Th(i_1)\xrightarrow{\widetilde{H}_2} T^{\wedge n}\xrightarrow{\rho^{\wedge n}} (S^{2,1})^{\wedge n}\xrightarrow{\Xi_n} S^{2n,n},
\]
where $\widetilde{H}_2$ is given by $\widetilde{H}_2(t,x_1,x_2,\dots,x_{n})=(x_1/t,x_2,\dots,x_{n})$, and $\Xi_n$ is the canonical shuffling isomorphism.

Identifying the first copy of $T$ with $\PP^1/\A^1$ via $\theta_2(x)= [x:1]$ we rewrite $\widetilde{H}_2$ as $\widetilde{H}_2=(\theta_2^{-1}\wedge id)\widetilde{H}_1$ with $\widetilde{H}_1$ given by
\[
\widetilde{H}_1(t,x_1,x_2,\dots,x_{n})=([x_1:t],x_2,\dots,x_{n}).
\]
Put $Y=(\A^1\times (\A^{n}-\{0\}))\cup (\{1\}\times \A^n)$ and consider the following diagram.
\[
\xymatrix{
(\PP^1/\A^1)\wedge T^{\wedge n-1} & T\wedge T^{\wedge n-1} \ar[l]_(0.4){\vartheta_2\wedge id}^(0.4){\simeq} \ar[r]^{\Xi_n\rho^{\wedge n}}_{\simeq} & S^{2n,n} \\
Th(i_1)\ar[u]^{\widetilde{H}_1} \ar[r]^(0.4){w_1}_(0.4){\simeq} & ((\A^2-\{0\})/ \A^1\times\Gm)\wedge T^{\wedge n-1} \ar[lu]_{\widetilde{H}_3\wedge id} \ar[r]^(0.76){w_2}_(0.76){\simeq} & Th(i) \ar[ld]_{j'}\\
(\A^2-\{0\},(1,1))\wedge T^{\wedge n-1}\ar[r]^(0.55){\psi_1}_(0.55){\simeq} \ar[ur]^{j\wedge id} & X/Y  & (X,x)\ar[u]^{z_1} \ar[l]_(0.35){\sigma_1}^(0.35){\simeq}
}
\]
Here $\widetilde{H}_3(x,y)=[y:x]$ and all the other maps are given by the tautological inclusions, i.e. $w_1$ is induced by the inclusion
$U=\Gm\times \A^n\subset (\A^2-\{0\})\times \A^{n-1}$,
$w_2$ and $\psi_1$ are induced by $(\A^2-\{0\})\times \A^{n-1}\subset X$,
$j'$ is given by the identity map on $X$ and $j$ is given by identity map on $\A^2-\{0\}$. Morphisms $w_1$ and $w_2$ are excision isomorphisms and $\psi_1$ is a composition of isomorphism $s_1\wedge id$ and excision isomorphism (see the next diagram), so it is an isomorphism as well. One can easily check that this diagram is commutative. Hence
\begin{multline*}
z_1^\pi (w^{\pi})^{-1}(th(U,\theta))=z_1^\pi ((w_2w_1)^{\pi})^{-1}(th(U,\theta))=\\
=\Sigma^{\infty}_T(\Xi_n \rho^{\wedge n}(\vartheta_2^{-1}\widetilde{H}_3j\wedge id)\psi_1^{-1}\sigma_1)
=(\psi_1^{-1}\sigma_1)^{\pi}\Sigma_T^{n-1}(\Sigma^{\infty}_T(\rho \vartheta\widetilde{H})),
\end{multline*}
with $\widetilde{H}=\widetilde{H}_3j$. There is the following commutative diagram consisting of isomorphisms.
\[
\xymatrix{
(\A^2-\{0\},(1,1))\wedge T^{\wedge n-1}\ar[r]^(0.55){\psi_1}_(0.55){\simeq}  \ar[d]_{s_1\wedge id}^{\simeq}& X/Y \\
(\A^2-\{0\})/((\A^1\times \Gm)\cup (\{1\}\times \A^1))\wedge T^{\wedge n-1}  \ar[ur]_{\simeq} & (\Gm,1)\wedge T\wedge T^{\wedge n-1} \ar[u]^{\sigma_2}_{\simeq} \ar[l]_(0.32){s_2\wedge id}^(0.32){\simeq}
}
\]
All the maps in the diagram are induced by the tautological inclusions, $s_2$ is induced by $\Gm\times \A^1\subset \A^2-\{0\}$ and $s_1$ is given by the identity map on $\A^2-\{0\}$. Morphisms $\sigma_2$, $s_2$ and the diagonal morphism are excision isomorphisms and $s_1$ is an isomorphism via the usual contraction argument.

Thus we  have 
\begin{multline*}
\Sigma^{1,0}((\psi_1^{-1}\sigma_1\sigma^{-1}\tau)^{\pi}\Sigma_T^{n-1}(\Sigma^{\infty}_T(\rho \vartheta\widetilde{H})))=\\
=\Sigma^{1,0}(((s_1^{-1}s_2\wedge id)\tau)^{\pi}\Sigma_T^{n-1}(\Sigma^{\infty}_T(\rho \vartheta\widetilde{H})))=\\
=\Sigma^{1,0}(\tau^{\pi}\Sigma_T^{n-1}((s_1^{-1}s_2)^{\pi}\Sigma^{\infty}_T(\rho \vartheta\widetilde{H}))).
\end{multline*}
To sum up, the above considerations together with Lemma~\ref{lem_tildeH} yield
\begin{multline*}
\partial i_\pi (1)=((id\wedge\rho)\tau_c^{-1})^{\pi}\Sigma^{1,0}(\tau^{\pi}\Sigma_T^{n-1}((s_1^{-1}s_2)^{\pi}\Sigma^{\infty}_T(\rho \vartheta\widetilde{H})))=\\
=\epsilon \cup ((id\wedge\rho)\tau_c^{-1})^{\pi}\Sigma^{1,0}(\tau^{\pi}\Sigma_T^{n}\Sigma^{1,1}\eta)=
\epsilon \cup ((\tau\wedge id)(id\wedge\rho)\tau_c^{-1})^{\pi}\Sigma^{1,0}\Sigma_T^{n}\Sigma^{1,1}\eta.
\end{multline*}
Now we examine the homomorphism
\[
\Theta=((\tau\wedge id)(id\wedge\rho)\tau_c^{-1})^{\pi}\Sigma^{1,0}\Sigma_T^{n}\Sigma^{1,1}\colon \pi^{-1,-1}(pt) \to \pi^{2n+1,n}(T^{\wedge n+1}).
\]
Unraveling the notation, this homomorphism can be represented as an external product with a $\Sigma_T^\infty$-suspension of the map $T^{\wedge n+1}\to S^{2n+2,n+1}$ corresponding to the following picture consisting of $\rho$-s and identity maps: 

\begin{tikzpicture}
    \tikzstyle{ann} = [draw=none,fill=none,right]
    \matrix (m) [matrix of math nodes, row sep=25, column sep=-6]
    { 
    \Gm & \wedge & \Gm & \wedge &\Gm & \wedge & \ldots & \wedge & \Gm  & \wedge  & S_s^1 & \wedge & S_s^1 & \wedge & \ldots & \wedge & S_s^1 & \wedge & S_s^1\\
    & \Gm & \wedge & \Gm & \wedge & S_s^1 & \wedge & \Gm &  \wedge &  S_s^1    & \wedge  & \ldots & \wedge & \Gm & \wedge & S_s^1  & \wedge & S_s^1   \\
    & \Gm & \wedge &  & T &  & \wedge &  &  T &    & \wedge  & \ldots & \wedge &  & T &  & \wedge & S_s^1   \\
        &  & T &  & \wedge &  &  T &    & \wedge  & \ldots & \wedge &  & T &  & \wedge & \Gm & \wedge & S_s^1   \\
        &  & T &  & \wedge &  &  T &    & \wedge  & \ldots & \wedge &  & T &  & \wedge &  & T &  \\
       };
   \path[->] (m-2-2.115) edge (m-1-1.260);
   \path[->] (m-2-4.120) edge (m-1-3.290);       
   \path[->] (m-2-8.125) edge (m-1-5.320);          
   \path[->] (m-2-14.130) edge (m-1-9.320);          
   \path[->] (m-2-6.60) edge (m-1-11.230);
   \path[->] (m-2-10.80) edge (m-1-13);   
   \path[->] (m-2-16.north) edge (m-1-17);      
   \path[->] (m-2-18.north) edge (m-1-19.260);    
   
   \path[->] (m-3-2) edge (m-2-2);
   \path[->] (m-3-5) edge (m-2-5);       
   \path[->] (m-3-9) edge (m-2-9);          
   \path[->] (m-3-15) edge (m-2-15);   
   \path[->] (m-3-18) edge (m-2-18);      
   
   \path[->] (m-4-3) edge (m-3-5);
   \path[->] (m-4-7) edge (m-3-9);       
   \path[->] (m-4-13) edge (m-3-15);          
   \path[->] (m-4-16.140) edge (m-3-2.330);   
   \path[->] (m-4-18) edge (m-3-18);      
   
   \path[->] (m-5-3.40) edge (m-4-17.240);
   \path[->] (m-5-7) edge (m-4-3);       
   \path[->] (m-5-10) edge (m-4-7);          
   \path[->] (m-5-13) edge (m-4-10);   
   \path[->] (m-5-17) edge (m-4-13);

  \node[draw=black, thick, dotted, inner sep=0em, fit=(m-2-4) (m-2-6)] {};
  \node[draw=black, thick, dotted, inner sep=0em, fit=(m-2-8) (m-2-10)] {};
  \node[draw=black, thick, dotted, inner sep=0em, fit=(m-2-14) (m-2-16)] {};  
  \node[draw=black, thick, dotted, inner sep=0em, fit=(m-4-16) (m-4-18)] {};
      
\end{tikzpicture}

\noindent
Here the first row of morphisms is the canonical shuffling isomorphism, the first row combined with the second one correspond to $\Sigma^{1,0}\Sigma_T^{n}\Sigma^{1,1}$, the third one is $\tau\wedge id$ and the fourth row is $(id\wedge \rho)\tau_c^{-1}$. Taking the composition we obtain the next picture (we write $\Gm\wedge S_s^1$ instead of $T$):

\begin{tikzpicture}
    \tikzstyle{ann} = [draw=none,fill=none,right]
    \matrix (m) [matrix of math nodes, row sep=30, column sep=-4]
    { \Gm & \wedge & \Gm & \wedge &\Gm & \wedge & \ldots & \wedge & \Gm & \wedge  & S_s^1 & \wedge & S_s^1 & \wedge & \ldots & \wedge & S_s^1 & \wedge & S_s^1 \\
       & \Gm & \wedge & S_s^1 & \wedge & \Gm & \wedge & S_s^1 & \wedge & \Gm & \wedge & S_s^1 & \wedge & \ldots & \wedge & \Gm & \wedge & S_s^1 & \\ };
   \path[->] (m-2-2.115) edge (m-1-1.260);
   \path[->] (m-2-6.120) edge (m-1-3.290);       
   \path[->] (m-2-10.125) edge (m-1-5.320);          
   \path[->] (m-2-16.130) edge (m-1-9.320);          
   \path[->] (m-2-4.80) edge (m-1-19.230);
   \path[->] (m-2-8.north) edge (m-1-11);   
   \path[->] (m-2-12.north) edge (m-1-13);      
   \path[->] (m-2-18.north) edge (m-1-17.260);    

  \node[draw=black, thick, dotted, inner sep=0em, fit=(m-2-2) (m-2-4)] {};
  \node[draw=black, thick, dotted, inner sep=0em, fit=(m-2-6) (m-2-8)] {};
  \node[draw=black, thick, dotted, inner sep=0em, fit=(m-2-10) (m-2-12)] {};  
  \node[draw=black, thick, dotted, inner sep=0em, fit=(m-2-16) (m-2-18)] {};
      
\end{tikzpicture}

\noindent The corresponding picture for $\Sigma_T^{n+1}$ looks as follows:

\begin{tikzpicture}
    \tikzstyle{ann} = [draw=none,fill=none,right]
    \matrix (m) [matrix of math nodes, row sep=30, column sep=-4]
    { \Gm & \wedge & \Gm & \wedge  & \ldots & \wedge &\Gm & \wedge & \Gm & \wedge  & S_s^1 & \wedge & S_s^1 & \wedge & \ldots & \wedge & S_s^1 & \wedge & S_s^1 \\
       & \Gm & \wedge & S_s^1 & \wedge & \Gm & \wedge & S_s^1 & \wedge & \ldots & \wedge & \Gm & \wedge & S_s^1 &  \wedge &\Gm & \wedge & S_s^1 & \\ };
       
   \path[->] (m-2-2.115) edge (m-1-1.260);
   \path[->] (m-2-6.120) edge (m-1-3.290);       
   \path[->] (m-2-12.125) edge (m-1-7.320);          
   \path[->] (m-2-16.130) edge (m-1-9.320);          
   \path[->] (m-2-4.80) edge (m-1-11.230);
   \path[->] (m-2-8.north) edge (m-1-13);   
   \path[->] (m-2-14.80) edge (m-1-17);      
   \path[->] (m-2-18.north) edge (m-1-19.260);    

  \node[draw=black, thick, dotted, inner sep=0em, fit=(m-2-2) (m-2-4)] {};
  \node[draw=black, thick, dotted, inner sep=0em, fit=(m-2-6) (m-2-8)] {};
  \node[draw=black, thick, dotted, inner sep=0em, fit=(m-2-12) (m-2-14)] {};  
  \node[draw=black, thick, dotted, inner sep=0em, fit=(m-2-16) (m-2-18)] {};
      
\end{tikzpicture}

\noindent
These pictures coincide up to a cyclic permutation of $S_s^1$-s. This permutation automorphism equals to $(-1)^{n}$ in the homotopy category, thus $\Theta=(-1)^{n}\Sigma_T^{n+1}$.
\end{proof}

\section{Inverting the stable Hopf map.}

Let $A^{*,*}(-)$ be a bigraded ring cohomology theory represented by a commutative monoid $A\in \SH(k)$. Inverting $\eta \in A^{-1,-1}(\pt)$ we obtain a new cohomology theory with $(2i,i)$ groups isomorphic to $(2i+n,i+n)$ ones by means of the cup product with $\eta^{-n}$. Put
\begin{gather*}
A^{n}(Y)=\left(A^{*,*}_\eta (Y)\right)^{n,0}= \left(A^{*,*}(Y)\otimes_{A^{*,*}(pt)}A^{*,*}(pt)[\eta^{-1}]\right)^{n,0},\\
A^{*}(Y)=\left(A^{*,*}_\eta (Y)\right)^{*,0}=\bigoplus_{n\in\mathbb{Z}} A^{n}(Y).
\end{gather*}

One can easily see that it is a cohomology theory. For the algebraic $K$-theory represented by $BGL$ \cite{PPR2} this construction gives $BGL^*(-)=0$ since we have $\eta\in BGL^{-1,-1}(pt)=K_{-1}(pt)=0$ and $BGL_{\eta}^{*,*}(-)=0$. As we will see in Corollary~\ref{cor_oriented} it is always the case that an oriented cohomology theory degenerates to a trivial cohomology theory. Thus we are interested in cohomology theories with a special linear orientation but without a general one. Our running example is hermitian $K$-theory represented by the spectrum $BO$ that derives to the Witt groups, i.e. for every smooth variety $X$ there is a natural isomorphism $BO^i(X)\cong W^i(X)$ (see \cite{An}).

For the stable cohomotopy groups there is the following result by Morel.
\begin{theorem}
\label{Morel}
There exists a canonical isomorphism
$\left(\pi^{*,*}_\eta(pt)\right)^{0,0}\xrightarrow{\simeq} W^0(pt)$.
\end{theorem}
\begin{proof}
See \cite{Mor2}.
\end{proof}

\begin{definition}
From now on $A^*(-)$ denotes a graded ring cohomology theory obtained via the above construction, i.e.
\[
A^{*}(Y)=\left(A^{*,*}_\eta (Y)\right)^{*,0}
\]
for a bigraded $SL$-oriented ring cohomology theory $A^{*,*}(-)$ represented by a commutative monoid $A\in \SH(k)$. We have Thom and Euler classes and all the machinery of $SL$-oriented theories, including the Gysin sequences and pushforwards. In order to stay in the chosen grading we need to modify the Thom and Euler classes as follows:
\[
th'(\Tc)=(-1)^{\frac{n(n-1)}{2}}th(\Tc)\cup\eta^{n}, \quad e'(\Tc)=(-1)^{\frac{n(n-1)}{2}}e(\Tc)\cup\eta^{n}
\]
for the special linear bundle $\Tc$ of rank $n$. The sign is introduced for the sake of multiplicativity of the characteristic classes. Shortening the notation we are going to omit the primes and write just $th(\Tc)$ and $e(\Tc)$ and refer to them as Thom and Euler classes. These classes are of degree $n$.
\end{definition}

\begin{rem}
\label{rem_epsilon}
Note that from $\epsilon$-commutativity we have $\eta\cup\eta=-\epsilon\cup(\eta\cup\eta)$, thus inverting $\eta$ we obtain $\epsilon=-1$ in $A^{*}(pt)$.
\end{rem}

\begin{definition}
Let $E$ be a vector bundle over a smooth variety $X$. The \textit{hyperbolic bundle associated to $E$} is the symplectic bundle 
\[
H(E)=\left(E\oplus E^\vee, \left(\begin{array}{cc}0 & 1\\ -1 & 0\end{array}\right)\right).
\]
Denote by $p_i(E)=(-1)^ib_{2i}(H(E))$ the signed even Borel classes of $H(E)$ and refer to them as \textit{Pontryagin classes}. The \textit{total Pontryagin class} is $p_*(E)=\sum p_i(E)t^{2i}$.
\end{definition}

\begin{rem}
This definition is parallel to the definition of the Pontryagin classes in topology with the Borel classes substituted for the Chern ones and using hyperbolisation instead of complexification. 
\end{rem}

We defined Pontryagin classes for arbitrary vector bundles without any additional structure. We will show later that for a special linear bundle $\Tc$ the odd Borel classes $b_{2i+1}(H(\Tc))$ vanish, so we are indeed interested only in the even ones. Also, for special linear bundles there is an interconnection between the top Pontryagin class and the Euler class. The following lemma shows it in the case of $\rank 2$ bundles and the general case would be dealt with in Corollary~\ref{cor_dual}.

\begin{lemma}
\label{lemm_borel_cl}
Let $\Tc=(E,\lambda)$ be a rank $2$ special linear bundle. Then
$$
b_* (H(E)) =1-e(\Tc)^2t^2,\quad p_*(\Tc)=1+e(\Tc)^2t^2.
$$
\end{lemma}
\begin{proof}
Let $\phi$ be the symplectic form on $E$ corresponding to $\lambda$. There exists an isomorphism \cite[Examples~1.1.21,~1.1.22]{Bal2}
$$
\left(E\oplus E^\vee, \left(\begin{array}{cc}0 & 1\\ -1 & 0\end{array}\right)\right)\cong \left(E\oplus E, \left(\begin{array}{cc}\phi & 0\\ 0 & -\phi\end{array}\right)\right),
$$
so we have
\begin{multline*}
b_* (H(E))=b_*(E,\phi)b_*(E,-\phi)=\\
=(1+b_1(E,\phi)t)(1+b_1(E,-\phi)t)=(1+e(E,\lambda)t)(1+e(E,-\lambda)t).
\end{multline*}
By Lemma~\ref{lemm_epsilon} and Remark~\ref{rem_epsilon} we have $e(E,-\lambda)=-e(E,\lambda)$, thus
\[
b_* (H(E))=(1+e(\Tc)t)(1-e(\Tc)t)=1-e(\Tc)^2t^2.
\]
In order to obtain the formula for the total Pontryagin class one should change the sign in front of $b_2(H(E))=-e(\Tc)^2$.
\end{proof}

\section{Complement to the zero section.}

In this section we compute the cohomology of the complement to the zero section of a special linear vector bundle. It turns out that there is a good answer in terms of the characteristic classes only in the case of the odd rank.

Recall that for a special linear bundle $\Tc$ we denote by $\Tc^0$ the complement to the zero section. We start from the following lemma concerning the case of a special linear bundle possessing a section.

\begin{definition}
We denote an operator of the $\cup$-product with an element by the symbol of the element writing $\alpha$ for $-\cup \alpha$.
\end{definition}

\begin{lemma}
\label{lem_E0_noncan}
Let $\Tc$ be a rank $k$ special linear bundle over a smooth variety $X$ with a nowhere vanishing section $s\colon X\to \Tc$. Then for some $\alpha\in A^{k-1}(\Tc^0)$ we have an isomorphism
$$
(1,\alpha)\colon A^*(X)\oplus A^{*+1-k}(X) \to A^*(\Tc^0).
$$
\end{lemma}
\begin{proof}
Consider the Gysin sequence
$$
{ }\to A^{*-k}(X)\xrightarrow{0} A^*(X)\xrightarrow{j^A} A^*(\Tc^0) \xrightarrow{\partial_A} A^{*-k+1}(X)\xrightarrow{0} {}
$$
The section $s$ induces a splitting $s^A$ for $j^A$ hence gives a splitting $r$ for $\partial_A$. We have the claim for $\alpha=r(1)$.
\end{proof}

We want to obtain an isomorphism which does not depend on the choice of the section, so we act as one acts in the projective bundle theorem for oriented cohomology theories: take a certain special linear bundle over $\Tc^0$ and compute its Euler class.

\begin{definition}
Let $p\colon E\to X$ be a vector bundle over a smooth variety $X$. The tautological line subbundle $L_E$ of $(p^\ast E)|_{E^0}$ could be trivialized by means of the diagonal section $\Delta\colon E^0\to E^0\times_X E$. Hence, by {Lemma \ref{quottriv}}, for a special linear bundle $(E,\lambda)$ there exists a canonical trivialization
$$
\lambda_{\Tc_E}\colon \det (p^\ast E|_{E^0}/L_E) \xrightarrow{\simeq} \triv_{E^0}.
$$
Thus we obtain a special linear bundle $\Tc_E=\left(\left(p^\ast E|_{E^0}/L_E\right),\lambda_{\Tc_E}\right)$ over $E^0$.
\label{defTc}
\end{definition}

For the Witt groups there is a result by Balmer and Gille.
\begin{theorem}
\label{Balmer}
Let $(E,\lambda)=(\triv^{2n+1}_\pt,1)$ be a trivialized special linear bundle of odd rank over a point with $n\ge 1$. Then for $e=e(\Tc_E)\in W^{2n}(E^0)$ we have an isomorphism
$$
(1,e)\colon W^*(pt)\oplus W^{*-2n}(pt)\xrightarrow{\simeq} W^*(E^0).
$$
\end{theorem}
\begin{proof}
See \cite[Theorem~8.13]{BG}.
\end{proof}

We can derive an analogous result for $A^{*}(-)$ from our computation in stable cohomotopy groups.
\begin{lemma}
\label{lem_main}
Let $(E,\lambda)=(\triv^{2n+1}_\pt,1)$, $n\ge 1$, be a trivialized special linear bundle over a point. Then for $e=e(\Tc_E)\in A^{2n}(E^0)$ we have an isomorphism
$$
(1,e)\colon A^*(pt)\oplus A^{*-2n}(pt)\xrightarrow{\simeq} A^*(E^0).
$$
\end{lemma}
\begin{proof}
Consider the Gysin sequence
$$
{\dots }\to A^{*-2n-1}(pt)\xrightarrow{0} A^*(pt)\to A^*(E^0) \xrightarrow{\partial_A} A^{*-2n}(pt)\xrightarrow{0} {\dots}
$$
The bundle $E$ is trivial hence $e(E,\lambda)=0$ and the Gysin sequence consists of short exact sequences.

Consider the dual special linear bundle $\Tc_E^\vee$. Taking the dual trivialization of $E^\vee$ we obtain
\[
\Tc_E^\vee=\{(x_0,\dots,x_{2n},y_0,\dots,y_{2n})\in E^0\times E^{\vee}\,|\, x_0y_0+\dots+x_{2n}y_{2n}=0\}.
\]
There is a section $s\colon E^0\to \Tc_E^\vee$ with
\[
s(x_0,x_1,x_2,\dots,x_{2n-1},x_{2n})=(x_0,x_1,\dots,x_{2n}, 0,x_2,-x_1,\dots, x_{2n}, -x_{2n-1}).
\]
This section meets the zero section in $\Gm \cong \left\{(t,0,\dots,0) \,|\, t\neq 0\right\}$.
{Proposition~\ref{prop_section}} states that $e(\Tc_E^\vee)=i_A(1)$ for the inclusion $i\colon\Gm\to\A^{2n+1}-\{0\}$ with the trivialization of $\det N_i$ arising from the trivialization of $\det \Tc_E^\vee$. Identify $N_i\cong \Tc_E^\vee|_\Gm$ with $U=\Gm\times \A^{2n}\subset E^0$ via
\[
(t,0,\dots,0,0,y_1\dots,y_{2n}) \mapsto (t,y_1,\dots,y_{2n}).
\]
The isomorphism $\lambda_{\Tc_E^\vee}\colon \det \Tc_E^\vee|_\Gm\xrightarrow{\simeq} \triv_\Gm$ arises from the canonical trivialization of $E^\vee|_\Gm$ and morphism $\phi\colon E^\vee|_{\Gm} \to L_E^\vee|_\Gm\cong \triv_\Gm$  with
\[
\phi(t,y_0,y_1,\dots,y_{2n})=(t,ty_0).
\]
Thus over $t$ for $\mathbf{y}^i=(y_{1}^i,y_{2}^i,\dots,y_{2n}^i)$ we have
\[
\lambda_{\Tc_E^\vee} (\mathbf{y}^1\wedge\mathbf{y}^2\wedge\dots\wedge\mathbf{y}^{2n})
=
\det \left(
\begin{array}{ccccc}
1/t & 0 & 0 & \hdots & 0 \\
0 & y_{1}^1 & y_{1}^2 & \hdots & y_{1}^{2n} \\
0 & y_{2}^1 & y_{2}^2 & \hdots & y_{2}^{2n} \\
\vdots & \vdots & \vdots & \ddots & \vdots \\
0 & y_{2n}^1 & y_{2n}^2 & \hdots & y_{2n}^{2n}
\end{array}\right)
\]
and $\theta \colon (U,\lambda_{\Tc_E^\vee})\xrightarrow{\simeq} (\triv_{\Gm}^{2n},1)$ with $\theta(t,y_1,y_2,\dots,y_{2n})=(t,y_1/t,y_2,\dots,y_{2n})$ is an isomorphism of special linear bundles.

Consider the following diagram with $i_\pi$ being a pushforward in stable cohomotopy groups for the closed embedding $i$ with the trivialization $\theta$ of the normal bundle.
$$
\xymatrix{
A^0(\Gm) \ar^{i_A}[rr] & & A^{2n}(E^0) \ar[rr]^{\partial_A}& & A^{0}(pt) \\
(\pi_\eta^{*,*}(\Gm))^{0,0} \ar[rr]^{i_\pi}\ar[u]& & (\pi_\eta^{*,*}(E^0))^{2n,0}  \ar[u] \ar[rr]^{\partial_\pi} & & (\pi^{*,*}_\eta(pt))^{0,0}  \ar[u]\\
}
$$
The left-hand side commutes since $\theta$ is an isomorphism of special linear bundles. The right-hand side of the diagram consist of the structure morphisms for $A^*$ and the boundary maps for the Gysin sequences of the inclusion $\{0\}\to E$ hence commutes as well. Proposition~\ref{prop_stable_coh} states that  $\partial_\pi i_\pi (1)=-1$, thus
\[
\partial_A(e(\Tc_E^\vee))=\partial_A i_A (1)=-1.
\]
Hence, examining the short exact sequences
\[
0\to A^*(pt)\to A^*(E^0) \xrightarrow{\partial_A} A^{*-2n}(pt)\to 0
\]
given by the Gysin sequence, we obtain that $\{1, e(\Tc_E^\vee)\}$ is a basis of $A^*(E^0)$ over $A^{*}(pt)$. 

There is a nowhere vanishing section of $\Tc_E^\vee\oplus\Tc_E^\vee$ constructed analogous to $s$ defined above,
so
\[
e(\Tc_E^\vee)^2=e(\Tc_E^\vee\oplus\Tc_E^\vee)=0.
\]
Lemma~\ref{lemm_dual_eu} yields that for some $\alpha_1,\alpha_2,\beta_1,\beta_2\in A^*(pt)$ we have
$$
e=(\alpha_1+\beta_1\cup e(\Tc_E^\vee))\cup e(\Tc_E^\vee)=\alpha_1\cup e(\Tc_E^\vee),
$$
\[
e(\Tc_E^\vee)=(\alpha_2+\beta_2\cup e(\Tc_E^\vee))\cup e=\alpha_2\cup \alpha_1 \cup e(\Tc_E^\vee).
\]
We already know that $\{1, e(\Tc_E^\vee)\}$ is a basis, thus $\alpha_2\cup \alpha_1=1$ and $\alpha_1$ is invertible. Hence $\{1, \alpha_1\cup e(\Tc_E^\vee)\}=\{1, e\}$ is a basis as well.

\end{proof}

\begin{cor}
\label{cor_oriented}
Let $A^{*,*}(-)$ be an oriented cohomology theory represented by a commutative monoid $A\in \SH(k)$. Then
$A^{*}(pt)=0$.
\end{cor}
\begin{proof}
There is a natural special linear orientation on $A^{*,*}(-)$ obtained by setting $th(E,\lambda)=th(E)$ with the latter Thom class arising from the orientation of $A^{*,*}(-)$. Hence for a rank $n$ special linear bundle we have $e(E,\lambda)=c_n(E)$. By the above lemma, for $E=\triv_{pt}^3$ there is an isomorphism
$$
(1,c_2(\Tc_E))\colon A^*(pt)\oplus A^{*-2}(pt)\xrightarrow{\simeq} A^*(E^0).
$$
Multiplicativity of total Chern classes yields $c_*(\triv_{E^0})c_*(\Tc_E)=c_*(\triv^3_{E^0})$, hence $c_2(\Tc_E)=0$. The above isomorphism yields $A^*(pt)=0$.
\end{proof}

Having a canonical basis for a trivial bundle we can glue it and obtain a basis for the cohomology of the complement to the zero section of an arbitrary special linear bundle of odd rank.
\begin{theorem}
\label{thm_main}
Let $(E,\lambda)$ be a special linear bundle of rank $2n+1,n\ge 1,$ over a smooth variety $X$. Then for $e=e(\Tc_E)$ we have an isomorphism
$$
(1,e)\colon A^*(X)\oplus A^{*-2n}(X)\to A^*(E^0).
$$
\end{theorem}
\begin{proof}
The general case is reduced to the case of the trivial vector bundle $E$ via the usual Mayer-Vietoris arguments. In the latter case we have a commutative diagram of the Gysin sequences
$$
\xymatrix{
0 \ar[r] & A^*(X) \ar[r]& A^*(E^0) \ar[r]^(0.50){\partial_A} & A^{*-2n}(X) \ar[r]& 0\\
0 \ar[r] & A^*(pt) \ar[r]\ar[u]& A^*(E'^0) \ar[r]^{\partial_A} \ar[u] & A^{*-2n}(pt)\ar[u]^{p^A} \ar[r]& 0
}
$$
with $E'=\triv_{pt}^{2n+1}$. By Lemma~\ref{lem_main} the element $\partial_A(e\left(\Tc_{E'}\right))$ generates $A^{*-2n}(pt)$ as a module over $A^{*}(pt)$, thus for a certain $\alpha\in A^{*}(pt)$ we have $\alpha\cup\partial_A(e\left(\Tc_{E'}\right))=1$. Using $E=p^*E'$ we obtain
$$
\alpha\cup\partial_A(e\left(\Tc_E\right))=\alpha\cup p^A\partial_A(e\left(\Tc_{E'}\right))=1,
$$
so $\partial_A\left(e\left(\Tc_E\right)\right)$ generates $A^{*-2n}(X)$ over $A^{*}(X)$. Hence $(1,e)$ is an isomorphism.
\end{proof}

\begin{rem}
In case of $\rank E=1$ one still has an isomorphism: a special linear bundle of rank one is a trivialized line bundle, hence there is an isomorphism
$$
A^*(X)\oplus A^*(X)\cong A^*(E^0)=A^*(X\times \Gm)
$$
induced by the isomorphism $A^*(pt)\oplus A^*(pt)\cong A^*(\Gm)$.
\end{rem}

\begin{cor}
Let $\Tc$ be a special linear bundle of odd rank over a smooth variety $X$. Then $e(\Tc)=0$.
\end{cor}
\begin{proof}
Set $\rank \Tc=2n+1$ and $e=e(\Tc)$. Consider the Gysin sequence
$$
{} \to  A^{0}(X) \xrightarrow{e} A^{2n+1}(X) \xrightarrow{j^A} A^{2n+1}(\Tc^0) \to A^{1}(X) \to {}
$$
The above calculations show that $j^A$ is injective hence $e=0$.
\end{proof}

\section{Special linear projective bundle theorem.}
In this section we obtain a special linear version of the projective bundle theorem. First of all we introduce the varieties that act as the projective spaces in the special linear context.

\begin{definition}
For $k<n$ consider the group
$$
P_k'=
\left(
\begin{array}{cc}
SL_k & *\\
0 & SL_{n-k}
\end{array}
\right).
$$
The quotient variety $SGr(k,n)=SL_n/P_k'$ is called a \textit{special linear Grassmann variety}. Put $SGr_X(k,n)=X\times SGr(k,n)$. We denote by $\Tc_1$ and $\Tc_2$ the \textit{tautological special linear bundles} over $SGr_X(k,n)$ with $\rank \Tc_1=k$ and $\rank \Tc_2 =n-k$.
\end{definition}
\begin{rem}
We have a projection $SL_n/P_k'\to SL_n/P_k$ identifying the special linear Grassmann variety with the complement to the zero section of the determinant of the tautological vector bundle over the ordinary Grassmann variety $Gr(k,n)$. This yields the following geometrical description of $SGr(k,n)$: fix a vector space $V$ of dimension $n$. Then
$$
SGr(k,n)=\{(U\le V, \lambda\in (\Lambda^k U)^0)\,|\, \dim U=k\}.
$$
In particular, we have $SGr(1,n)\cong \A^{n}-\{0\}$.
\end{rem}

\begin{theorem}
\label{theorem_proj_b}
For a smooth variety $X$ we have the following isomorphisms.
$$
(1,e_1,...,e_1^{2n-2},e_2)\colon \bigoplus\limits_{i=0}^{2n-2} A^{*-2i}(X)\oplus A^{*-2n+2}(X)\to A^*(SGr_X(2,2n)),
$$
$$
(1,e_1,e_1^2,...,e_1^{2n-1})\colon \bigoplus\limits_{i=0}^{2n-1} A^{*-2i}(X)\to A^*(SGr_X(2,2n+1)),
$$
with $e_1=e(\Tc_1)$, $e_2=e(\Tc_2)$.
\end{theorem}

\begin{proof}
We are going to deal with several special linear Grassmann varieties at once, so we will use  $\Tc_i(r,k)$ for $\Tc_i$ over $SGr_X(r,k)$ and abbreviate $e(\Tc_i(r,k))$ to $e_i(r,k)$ and $e(\Tc_i(r,k)^\vee)$ to $e_i^\vee(r,k)$.
The proof is done by induction on the Grassmannian's dimension.

\textbf{The base case.}
We have $SGr_X(2,3)\cong SGr_X(1,3)\cong \A^3_X-\{0\}$ and under these isomorphisms the bundle $\Tc_1(2,3)^\vee$ goes to $\Tc_2(1,3)$ which goes to $\Tc_{\triv_{_X}^3}$ in the notation of definition~\ref{defTc}. Note that $\rank \Tc_1(2,3)=2$, thus $\Tc_1(2,3)\cong \Tc_1(2,3)^\vee$ and $e(\Tc_1(2,3))=e\left(\Tc_1(2,3)^\vee\right)$. Hence {Theorem~\ref{thm_main}} gives the claim for $SGr_X(2,3)$.

\textbf{Basic geometry.}
Fix a vector space $V$ of dimension $k+1$, a subspace $W\le V$ of codimension one and forms $\mu_1\in (\Lambda^{k+1}V)^0,\mu_2\in (\Lambda^{k}W)^0$. Then we have the following diagram constructed in the same vein as in the case of ordinary Grassmannians: 
$$
\xymatrix{
SGr(2,k)\ar[r]^(0.43){i} & SGr(2,k+1) &Y \ar[l]_(0.3){j} \ar[d]^{p}\\
& & SGr(1,k)
}
$$
the inclusion $i$ corresponds to the pairs $(U,\mu\in(\Lambda^2U)^0)$ with $U\le W$, $\dim U=2$; the open complement $Y$ consists of the pairs $(U,\mu\in(\Lambda^2U)^0)$ with $\dim U=2, \dim U\cap W=1$; the projection $p$ is given by $p(U,\mu)=(U\cap W, \mu')$ where $\mu'$ is given by the isomorphisms $\Lambda^kW\otimes V/W\cong \Lambda^{k+1}V$ and $\Lambda^{k-1}(U\cap W)\otimes V/W\cong \Lambda^2 U$ and forms $\mu_1,\mu_2$ and $\mu$. Here $i$ is a closed embedding, $j$ is an open embedding and $p$ is an $\A^k$-bundle. Take an arbitrary $f\in V^\vee$ such that $\ker f=W$. It gives rise to a constant section of the trivial bundle $\left(\triv_{SGr(2,k+1)}^{k+1}\right)^\vee$ hence a section of $\Tc_1(2,k+1)^\vee$. The latter section vanishes exactly over $i(SGr(2,k))$. Note that we have $\rank \Tc_1(2,k+1) =2$ hence $e^\vee_1(2,k+1)=e_1(2,k+1)$.

\textbf{k=2n-1.}
Identify $A^*(X\times Y)\cong A^*(SGr_X(1,2n-1))$ via $p^A$ and consider the localization sequence.
\[
\to A^{*-2}(SGr_X(2,2n-1))\xrightarrow{i_A} A^*(SGr_X(2,2n))\xrightarrow{j^A} A^*(SGr_X(1,2n-1))\to
\]
Theorem~\ref{thm_main} states that $\{1,e_2(1,2n-1)\}$ is a basis of $A^*(SGr_X(1,2n-1))$ over $A^*(X)$. We have $j^*\Tc_2(2,2n)\cong p^*\Tc_2(1,2n-1)$ and
$$
j^A(e_2(2,2n))=e_2(1,2n-1)
$$
hence $j^A$ is a split surjection (over $A^*(X)$) with the splitting defined by
$$
1\mapsto 1,\, e_2(1,2n-1)\mapsto e_2(2,2n).
$$
Then $i_A$ is injective. Hence to obtain a basis of $A^*(SGr_X(2,2n))$ it is sufficient to calculate the pushforward for a basis of $A^{*-2}(SGr_X(2,2n-1))$ and combine it with $\{1,e_2(2,2n)\}$. Using the induction we know that
$$
\{1,e_1(2,2n-1),\dots,e_1(2,2n-1)^{2n-3}\}
$$
is a basis of  $A^{*}(SGr_X(2,2n-1))$. We have $i^*(\Tc_1(2,2n))\cong \Tc_1(2,2n-1)$ hence
$$
e_1(2,2n-1)=i^A(e_1(2,2n)).
$$
By Proposition~\ref{prop_section} we have
$$
i_A(e_1(2,2n-1)^l)=e_1(2,2n)^{l+1}
$$
obtaining the desired basis
$$
\{e_1(2,2n),e_1(2,2n)^2,\dots,e_1(2,2n)^{2n-2},1,e_2(2,2n)\}
$$ of $A^*(SGr_X(2,2n))$ over $A^*(X)$.

\textbf{k=2n.} Again identify $A^*(X\times Y)\cong A^*(SGr_X(1,2n))$ via $p^A$ and consider the localization sequence.
$$
\xrightarrow{\partial_A} A^{*-2}(SGr_X(2,2n))\xrightarrow{i_A} A^*(SGr_X(2,2n+1))\xrightarrow{j^A} A^*(SGr_X(1,2n))\xrightarrow{\partial_A}
$$
Using the induction we know a basis of $A^{*}(SGr_X(2,2n))$, namely
$$
\{1,e_1(2,2n),e_1(2,2n)^2,\dots,e_1(2,2n)^{2n-2},e_2(2,2n)\}
$$
and Lemma~\ref{lem_E0_noncan} gives us a non-canonical basis $\{1,\alpha\}$ for $A^*(SGr_X(1,2n))$. Examine $i_A(e_2(2,2n))$. It can't be computed using Proposition~\ref{prop_section} since it seems that $e_2(2,2n)$ can not be pullbacked from $A^*(SGr_X(2,2n+1))$, so we use the following argument. Consider a nonzero vector $w\in W$. It induces constant sections of $\triv^{2n}_{SGr(2,2n)}$ and $\triv^{2n+1}_{SGr(2,2n+1)}$ and sections of $\Tc_2(2,2n)$ and $\Tc_2(2,2n+1)$. The latter sections vanish over $SGr_X(1,2n-1)$ and $SGr_X(1,2n)$ respectively. Here $SGr_X(1,2n-1)$ corresponds to the vectors in $W/\langle w\rangle$ and $SGr_X(1,2n)$ corresponds to the vectors in $V/\langle w\rangle $. Hence we have the following commutative diagram consisting of closed embeddings.
$$
\xymatrix{
SGr_X(1,2n-1) \ar[r]^{r'} \ar[d]^{i'} & SGr_X(2,2n) \ar[d]^{i} \\
SGr_X(1,2n) \ar[r]^{r} & SGr_X(2,2n+1)
}
$$
By Proposition~\ref{prop_section} we have $e_2(2,2n)=r'_A(1)$, so, using Proposition~\ref{prop_composition}, we obtain
$i_A(e_2(2,2n))=r_Ai'_A(1)$. Notice that $N_{i'}$ is a trivial bundle of rank one. In fact, there is a section of trivial bundle $\Tc_1(1,2n)^\vee$ over $SGr_X(1,2n)$ constructed using the same element $f$ such that $\ker f=W$ and this section meets the zero section exactly at $SGr_X(1,2n-1)$. So we have
$$
i_A(e_2(2,2n))=r_Ai'_A(1)=r_A(e_1^\vee(1,2n))=r_A(0)=0.
$$

We claim that $\ker i_A=A^*(X)\cup e_2(2,2n)$ and $\operatorname{Im} j^A =A^*(X)\cup 1$. We have $j^A(1)=1$, hence $\partial_A(1)=0$ and
$$
\ker i_A=\operatorname{Im} \partial_A = A^*(X)\cup\partial_A(\alpha).
$$
The localization sequence is exact, so we have
$$
e_2(2,2n)=\partial_A(y\cup\alpha)=y\cup\partial_A(\alpha)
$$
for some $y\in A^*(X)$ and since $e_2(2,2n)$ is an element of the basis, $y$ is not a zero divisor. Consider the presentation of $\partial_A(\alpha)$ with respect to the chosen basis:
$$
\partial_A(\alpha)=x_0\cup 1+x_1\cup e_1(2,2n)+\dots +x_{2n-2}\cup e_1(2,2n)^{2n-2}+z\cup e_2(2,2n).
$$
We have $y\cup\partial_A(\alpha)=e_2(2,2n)$, hence $y\cup z=1$ and every $y\cup x_i=0$, hence $x_i=0$. Then $\partial_A(\alpha)=z\cup e_2(2,2n)$ and
$$
\ker i_A=\operatorname{Im} \partial_A=A^*(X)\cup\partial_A(\alpha) =A^*(X)\cup e_2(2,2n).
$$
We have
$$
\partial_A(x_0\cup 1+x_1 \cup \alpha)=x_1\cup\partial_A(\alpha)=x_1\cup z\cup e_2(2,2n),
$$
hence $\operatorname{Im} j^A=\ker \partial_A=A^*(X)\cup 1$.

There is an obvious splitting for $A^*(SGr_X(2,2n+1))\xrightarrow{j^A} \operatorname{Im} j^A$, $1\mapsto 1$. Then
calculating by the same vein as in the odd-dimensional case the pushforwards for the basis of $\operatorname{Coker} \partial_A$, $\{e_1(2,2n)^l\}$,  and adding to them $\{1\}$, we obtain the desired basis of $SGr_X(2,2n+1)$
\[
\{e_1(2,2n+1),\dots,e_1(2,2n+1)^{2n-1},1\}. \qedhere
\]
\end{proof}

\begin{definition}
Let $\Tc$ be a rank $n$ special linear bundle over a smooth variety $X$. We define the \textit{relative special linear Grassmann variety} $SGr(k,\Tc)$ twisting $SGr_X(k,n)$ with the cocycle $\gamma\in H^1(X,SL_n)$ given by $\Tc$. This variety is an $SGr(k,n)$-bundle over $X$. Similarly to the above, we denote by $\Tc_1$ and $\Tc_2$ the tautological special linear bundles over $SGr(k,\Tc)$.
\end{definition}

\begin{theorem}
\label{theorem_proj_b2}
Let $\Tc$ be a special linear bundle over a smooth variety $X$.
\begin{enumerate}
\item
If $\rank \Tc=2n$ then there is an isomorphism
$$
(1,e_1,...,e_1^{2n-2},e_2)\colon \bigoplus\limits_{i=0}^{2n-2} A^{*-2i}(X)\oplus A^{*-2n+2}(X)\xrightarrow{\simeq} A^*(SGr(2,\Tc)),
$$
with $e_1=e(\Tc_1)$, $e_2=e(\Tc_2)$.
\item
If $\rank \Tc=2n+1$ then there is an isomorphism
$$
(1,e,e^2,...,e^{2n-1})\colon \bigoplus\limits_{i=0}^{2n-1} A^{*-2i}(X)\xrightarrow{\simeq} A^*(SGr(2,\Tc)),
$$
with $e=e(\Tc_1)$.
\end{enumerate}
\end{theorem}
\begin{proof}
The general case is reduced to the case of the trivial bundle $\Tc$ via the usual Mayer-Vietoris arguments. The latter case follows from Theorem~\ref{theorem_proj_b}.
\end{proof}

\section{A splitting principle.}
In this section we assert a splitting principle for $SL$-oriented cohomology theories with inverted stable Hopf map. The principle states that from the viewpoint of such cohomology theories every special linear bundle is a direct sum of rank 2 special linear bundles and at most one trivial linear bundle.

\begin{definition}
For $k_1<k_2<\dots<k_m$ consider the group
$$
P_{k_1,\dots, k_{m-1}}'=
\left(
\begin{array}{cccc}
SL_{k_1} & * & \hdots & *\\
0 & SL_{k_2-k_1} & \hdots & *\\
\vdots & \vdots & \ddots & \vdots\\
0 & 0&\hdots & SL_{k_{m}-k_{m-1}}
\end{array}
\right)
$$
and define a \textit{special linear flag variety} as the quotient
$$
\SF(k_1,\dots,k_m)=SL_{k_m}/P_{k_1,\dots, k_{m-1}}'.
$$
In particular, we are interested in the following varieties:
$$
\SF(2n)=\SF(2,4,...,2n),\, \SF(2n+1)=\SF(2,4,...,2n,2n+1).
$$
These varieties are called \textit{maximal $SL_2$ flag varieties}. Similar to the case of the special linear Grassmannians we denote by $\Tc_i$ the tautological special linear bundles over $\SF(k_1,k_2,\dots,k_m)$ with $\rank \Tc_i=k_i-k_{i-1}$.

\end{definition}

\begin{rem}
The projection
$$
\SF(k_1,k_2,\dots,k_m)=SL_{n}/P_{k_1,\dots, k_m}'\to SL_{n}/P_{k_1,\dots, k_m}=\mathcal{F}(k_1,k_2,\dots,k_m)
$$
yields the following geometrical description of the special linear flag varieties. Consider a vector space $V$ of dimension $k_m$. Then we have
\begin{multline*}
\SF(k_1,k_2,\dots,k_m)=\\
=\left\{(V_1\le\dots\le V_{m-1}\le V,\, \lambda_1,\dots,\lambda_{m-1})\,\left|\, \dim V_j=k_j, \, \lambda_j\in \left(\Lambda^{k_j} V_j\right)^0 \right. \right\}
\end{multline*}

\end{rem}

\begin{definition}
Let $\Tc$ be a rank $n$ special linear bundle over a smooth variety $X$. We define the \textit{relative special linear flag variety} $\SF(k_1,k_2,\dots,k_{m-1},\Tc)$ twisting $X\times \SF(k_1,k_2,\dots,k_{m-1},n)$ with the cocycle given by $\Tc$. This variety is an $\SF(k_1,k_2,\dots,k_{m-1},n)$-bundle over $X$. We denote relative version of the maximal $SL_2$ flag variety by $\SF(\Tc)$.
\end{definition}

\begin{theorem}
\label{thm_split}
Let $\Tc$ be a rank $k$ special linear bundle over a smooth variety $X$. Then $A^*(\SF(2,4,\dots,2n,\Tc))$ is a free module over $A^*(X)$ with the following basis:
\begin{itemize}
\item
$k$ is odd:
$$\left\{\left.e_1^{m_1}e_2^{m_2}\dots e_n^{m_n}\,\right|\, 0\le m_i\le k-2i\right\},$$
\item
$k$ is even:
$$\left\{u_1u_2\dots u_{n} \,\left|\, u_i=
\left[\begin{array}{l}e_i^{m_i},\, 0\le m_i\le k-2i \\  e_{i+1}e_{i+2}\dots e_{n+1} \end{array}\right. \right.\right\},$$
\end{itemize}
where $e_i=e(\Tc_i,\lambda_{\Tc_i})$.
\end{theorem}
\begin{proof}
Proceed by induction on $n$. For $n=1$ the claim follows from Theorem~\ref{theorem_proj_b2}.

Consider the projection
$$
p\colon Y=\SF(2,4,\dots,2n,\Tc)\to \SF(2,4,\dots,2n-2,\Tc)=Y_1
$$
that forgets about the last subspace. Denote the tautological bundles over $Y$ by $\Tc_i$ and the tautological bundles over $Y_1$ by $\Tc_i '$.

\textbf{k is odd.}
 Using an isomorphism $ Y \cong SGr(2,\Tc_n')$ and Theorem~\ref{theorem_proj_b2} we obtain that $A^*(Y)$ is a free module over $A^*(Y_1)$ with the basis
$$
\mathcal{B}=\left\{1,e_n,\dots, e_n^{k-2n}\right\}.
$$
Using the induction we have the following basis for $A^*(Y_1)$:
$$
\mathcal{B}_1=\left\{\left.e_1'^{m_1}e_2'^{m_2}\dots e_{n-1}'^{m_{n-1}}\,\right|\, 0\le m_i\le k-2i\right\},
$$
with $e_i'=e(\Tc_i')$. One has $p^*(\Tc_i')\cong \Tc_i$ and $p^A(e_i')=e_i$ for $i\le n-1$. Computing the pullback for $\mathcal{B}_1$ and multiplying it with $\mathcal{B}$ we obtain the desired basis.

 \textbf{k is even.}
 This case is completely analogous to the previous one.
 We have an isomorphism $ Y \cong SGr(2,\Tc_n')$. By Theorem~\ref{theorem_proj_b2} we know that $A^*(Y)$ is a free module over $A^*(Y_1)$ with the basis
 $$
\mathcal{B}=\left\{u_{n} \,\left|\, u_n=
\left[\begin{array}{l}e_n^{m_n},\, 0\le m_n\le k-2n \\ e_{n+1} \end{array}\right. \right.\right\}.
 $$

Using the induction we have the following basis for $A^*(Y_1)$:
$$\mathcal{B}_1=\left\{u_1u_2\dots u_{n-1} \,\left|\, u_i=
\left[\begin{array}{l}e_i'^{m_i},\, 0\le m_i\le k-2i \\  e_{i+1}'e_{i+2}'\dots e_{n}' \end{array}\right. \right.\right\}$$
with $e_i'=e(\Tc_i')$. Note that $p^*(\Tc_i')\cong \Tc_i$ and $p^A(e_i')=e_i$ for $i\le n-1$. Multiplicativity of the Euler classes yields $p^A(e_n')=e_ne_{n+1}$. Computing the pullback for $\mathcal{B}_1$ and multiplying it with $\mathcal{B}$ we obtain the desired basis of $Y$.
\end{proof}

A straightforward consequence of the splitting principle is the following corollary relating the top characteristic classes of special linear bundles.

\begin{cor}
\label{cor_dual}
Let $\Tc$ be a special linear bundle over a smooth variety $X$ and put $n=[\frac{1}{2}\rank \Tc]$. Then we have
\begin{enumerate}
\item
$
e(\Tc)=e(\Tc^\vee),
$
\item
$
b_{2i+1}(H(\Tc))=0
$ for all $i$,
\item
$p_{i}(\Tc)=0$ for $i>n$,
\item
if $\rank \Tc =2n$ then $p_n(\Tc)=e(\Tc)^2$.
\end{enumerate}
\end{cor}
\begin{proof}
Consider $r\colon Y\xrightarrow{r_1}\SF(\Tc)\xrightarrow{r_2} X$ with $r_1$ being an $\A^s$-bundle splitting the $r^*\Tc$ into a sum of special linear vector bundles isomorphic to $r_1^*\Tc_i$. One can construct $Y$ in a similar way as we used in Lemma~\ref{lem_eulermult}. From Theorem~\ref{thm_split} we know that $r^A$ is an injection. We have $r^*\Tc\cong \bigoplus_i r_1^*\Tc_i$ and $r^*\Tc^\vee\cong \bigoplus_i r_1^*\Tc_i^\vee$. Note that $\rank r_1^*\Tc_i \le 2$, hence $r_1^*\Tc_i\cong r_1^*\Tc_i^\vee$ and we obtain $r^*\Tc\cong r^*\Tc^\vee$, so $r^Ae(\Tc)=r^Ae(\Tc^\vee)$ and $e(\Tc)=e(\Tc^\vee)$. 

By Lemma~\ref{lemm_borel_cl} and multiplicativity of total Borel classes we have
$$
b_*(r^*H(\Tc))=b_*(H(r^*\Tc))=\prod_{i=1}^n (1-e(r^*_1\Tc_i)^2t^2),
$$
thus the odd Borel classes vanish and $p_i=0$ for $i>n$. Moreover, for a special linear bundle of the even rank this equality yields
\[
r^Ap_{n}(\Tc)=(-1)^nr^Ab_{2n}(\Tc)=(-1)^{2n}\prod_{i=1}^n e(r_1^*\Tc_i)^2=(r^Ae(\Tc))^2. \qedhere
\]
\end{proof}

Another consequence of the splitting principle is multiplicativity of total Pontryagin classes.
\begin{lemma}
\label{lem_BorelCartan}
Let $\Tc$ be a special linear bundle over a smooth variety $X$ and let $\Tc_1\le \Tc$ be a special linear subbundle. Then $p_*(\Tc)=p_*(\Tc_1)p_*(\Tc/\Tc_1)$.
\end{lemma}
\begin{proof}
Considering the $A^r$-bundle $p\colon Y\to X$ described in Lemma~\ref{lem_eulermult} one may assume that $\Tc\cong \Tc_1\oplus \Tc/\Tc_1$. The claim of the lemma follows from the second item of the above corollary and multiplicativity of total Borel classes.
\end{proof}

We finish this section with the theorem claiming that every special linear bundle of even rank is cohomologically symplectic in the following precise sense.

\begin{theorem}
Let $\Tc=(E,\lambda)$ be a special linear bundle of even rank over a smooth variety $X$. Then there exists a morphism of smooth varieties $p\colon Y\to X$ such that $A^*(Y)$ is a free $A^*(X)$-module (via $p^A$) and $p^*E$ has a canonical symplectic form $\phi$ compatible with trivialization $p^*\lambda$.
\end{theorem}
\begin{proof}
Consider the same morphism $p\colon Y\to X$ as we used in the proof of Corollary~\ref{cor_dual}, i.e. $Y$ is an $\A^r$-bundle over $\SF(\Tc)$ such that 
\[
p^*\Tc\cong \bigoplus \Tc_i,
\]
where $\Tc_i$ are special linear bundles of rank two. Theorem~\ref{thm_split} yields $A^*(Y)$ is a free $A^*(X)$-module. The special linear bundles $\Tc_i=(E,\lambda_i)$ have canonical symplectic forms $\phi_i$ induced by trivializations $\lambda_i$. Hence $E\cong \bigoplus E_i$ has a symplectic form $\phi=\phi_1\perp\phi_2\perp\ldots\perp\phi_n$ and it is compatible with $\lambda=\lambda_1\otimes\lambda_2\otimes\ldots\otimes\lambda_n$.
\end{proof}

\section{Cohomology of the partial flags.}
Now we turn to the computation of the relations that Pontryagin and Euler classes of the tautological bundles satisfy.

\begin{theorem}
\label{thm_SGr(2,Tc)}
Let $\Tc$ be a special linear bundle over a smooth variety $X$. Put $e=e(\Tc), p_i=p_i(\Tc)$.
\begin{enumerate}
\item
If $\rank \Tc = 2n$ then
$$
\phi_1\colon\bigslant{A^*(X)[e_1,e_2]}{R_{2,2n}} \xrightarrow{} A^*(SGr(2,\Tc)),
$$
where 
\[
R_{2,2n}=\big(e_1e_2-e, e_2^2+\sum_{i=0}^{n-1} (-1)^{i-1}p_{n-i-1}e_1^{2i} \big)
\]
and the homomorphism is induced by $\phi_1(e_1)=e(\Tc_1),\phi_1(e_2)=e(\Tc_2)$, is an isomorphism of $A^*(X)$-algebras.
\item
If $\rank \Tc = 2n+1$ then
$$
\phi_2\colon\bigslant{A^*(X)[e_1]}{\big( \sum_{i=0}^n (-1)^ip_{n-i}e_1^{2i} \big)} \xrightarrow{} A^*(SGr(2,\Tc)),
$$
induced by $\phi_2(e_1)=e(\Tc_1)$, is an isomorphism of $A^*(X)$-algebras.
\end{enumerate}
\end{theorem}
\begin{proof}
In view of the Theorem~\ref{theorem_proj_b2} it is sufficient to show that the claimed relations hold, since $\phi_1$ and $\phi_2$ are supposed to map the bases to the bases. We have an isomorphism $p^*\Tc/\Tc_1\cong \Tc_2$ for the natural projection $p\colon SGr(2,\Tc)\to X$, so the relation $e(\Tc_1)e(\Tc_2)=e$ follows from the multiplicativity of the Euler class. In order to obtain the others relation compute the total Pontryagin class
\[
p_*(p^A\Tc)=p_*(\Tc_1)p_*(\Tc_2).
\]
Dividing by $p_*(\Tc_1)=1+p_1(\Tc_1)t^2$ in $A^{*}(SGr(2,\Tc))[[t]]$ we get
\[
1+(p_1-p_1(\Tc_1))t^2+\ldots+ \left(\sum_{j=0}^{i} (-1)^{j}p_{i-j}p_1(\Tc_1)^{j}\right) t^{2i}+\ldots =
\sum p_i(\Tc_2)t^{2i}.
\]
In case of $\rank \Tc=2n$ recall that by Corollary~\ref{cor_dual} 
\[
p_1(\Tc_1)=e(\Tc_1)^2, \quad p_{n-1}(\Tc_2)=e(\Tc_2)^2
\]
and compare the coefficients at $t^{2n-2}$. In the other case by the same Corollary~\ref{cor_dual} we have $p_1(\Tc_1)=e(\Tc_1)^2$ and $p_n(\Tc_2)=0$, so the claim follows from the comparison of the coefficients at $t^{2n}$.
\end{proof}

\begin{cor}
\label{cor_SGr(2,k)}
Homomorphisms of $A^*(pt)$-algebras
\begin{enumerate}
\item
$\phi_1\colon\bigslant{A^*(pt)[e_1,e_2]}{\big( e_1e_2,\,e_1^{2n-2}+(-1)^ne_2^2\big)} \xrightarrow{\simeq} A^*(SGr(2,2n))$,
\item
$\phi_2\colon\bigslant{A^*(pt)[e_1]}{\big( e_1^{2n}\big)} \xrightarrow{\simeq} A^*(SGr(2,2n+1))$,
\end{enumerate}
induced by $\phi_1(e_1)=e(\Tc_1)$, $\phi_1(e_2)=e(\Tc_2)$ and $\phi_2(e_1)=e(\Tc_1)$ are isomorphism.
\end{cor}
\begin{proof}
The characteristic classes of the trivial bundle vanish, so the claim follows from the above theorem.
\end{proof}

In order to write down the relations for the cohomology of the special linear flag varieties we need to perform certain computations involving symmetric polynomials. Put $h_i(x_1,x_2,\ldots,x_n)$ for the $i$-th complete symmetric polynomial in $n$ variables. 

\begin{lemma}
\label{lem_borelcompletepoly}
Let $\triv_{X}^n$ be a trivialized special linear bundle over a smooth variety $X$. Suppose that there is an isomorphism of special linear bundles $(\triv_{X}^n,1) \cong  (\oplus_{i=1}^k \Tc_i) \oplus \Tc'$ for the special linear bundles $\Tc_i$ of rank $2$. Then
\[
p_i(\Tc')=(-1)^ih_i(e(\Tc_1)^2,e(\Tc_2)^2,\ldots,e(\Tc_k)^2).
\]
\end{lemma}
\begin{proof}
Using the multiplicativity of total Pontryagin classes and Lemma~\ref{lemm_borel_cl} we obtain
\[
(\prod_{i=1}^k(1+e(\Tc_i)^2t^2))p_*(\Tc')=p_*(\triv_{X}^n)=1.
\]
The claim follows from the comparison of the coefficients at $t^{2i}$ in the series obtained inverting $(1+e(\Tc_i)^2t^2)$ in $A^*(X)[[t]]$.
\begin{multline*}
1+p_1(\Tc')t^2+p_2(\Tc')t^4+\ldots = \prod_{i=1}^{k} (1-e(\Tc_i)^2t^2+e(\Tc_i)^4t^4-\ldots)=\\
=1-h_1(e(\Tc_1)^2,\ldots,e(\Tc_k)^2)t^2+h_2(e(\Tc_1)^2,\ldots,e(\Tc_k)^2)t^4-\ldots 
\end{multline*}
\end{proof}

\begin{proposition}
\label{prop_partialFlag}
We have the following isomorphisms of $A^*(pt)$-algebras.
\begin{enumerate}
\item
$
\phi_1\colon\bigslant{A^*(pt)[e_1,\ldots,e_m,e'_m]}{I_{2m,2n}} \xrightarrow{\simeq} A^*(\SF(2,4,\ldots,2m,2n)),
$\\
where
\begin{align*}
I_{2m,2n}=\big(& e_1e_2\ldots e_me_m', (-1)^{n}e_2^2\ldots e_m^2e_m'^2+h_{n-1}(e_1^2), \\
&(-1)^{n-1}e_3^2\ldots e_m^2e_m'^2+h_{n-2}(e_1^2,e_2^2),\ldots,\\
 &(-1)^{n-m+1}e_m'^2+ h_{n-m}(e_1^2,e_2^2,\ldots,e_m^2)\big)
\end{align*}
and the isomorphism is induced by $\phi_1(e_i)=e(\Tc_i)$, $\phi_1(e_m')=e(\Tc_{m+1})$.
\item
$
\phi_2\colon\bigslant{A^*(pt)[e_1,\ldots,e_m]}{I_{2m,2n+1}} \xrightarrow{\simeq} A^*(\SF(2,4,\ldots,2m,2n+1)),
$\\
where
\[
I_{2m,2n+1}=\big(h_n(e_1^2), h_{n-1}(e_1^2,e_2^2),\ldots,h_{n-m+1}(e_1^2,e_2^2,\ldots,e_m^2) \big)
\]
and the isomorphism is induced by $\phi_2(e_i)=e(\Tc_i)$.
\end{enumerate}
\end{proposition}
\begin{proof} The detailed proof would be quite messy, so we present the reasoning only for item (2). The even case is quite the same, but the formulas are a bit more complicated.

The considered special linear flag varieties are iterated $SGr(2,k)$-bundles, so one may proceed by induction on $m$. The case of $m=1$ was dealt with in Corollary~\ref{cor_SGr(2,k)}. Put $\SF_m=\SF(2,4,\ldots,2m,2n+1)$ and write $\Tc_{i,m}$ for the $i$-th tautological special linear bundle over $\SF_m$. As usual, taking an $\mathbb{A}^r$-bundle $Y\to \SF_m$ allows us to assume that $\triv^{2n+1}_{\SF_m}$ splits into a sum $(\oplus_{i=1}^m \Tc_{i,m})\oplus \Tc_{m+1,m}$. Put 
\[
\Tc_{j,m}'=\bigoplus\limits_{i=j+1}^{m+1} \Tc_{i,m}
\] 
and shorten the notation for the Euler classes via $e_{i,m}=e(\Tc_{i,m}), e'_{i,m}=e(\Tc_{i,m}')$.

There is a natural isomorphism $\SF_{m+1}\cong SGr(2,\Tc_{m+1,m})$. By Theorem~\ref{thm_SGr(2,Tc)} there is an isomorphism
\[
\bigslant{A^*(\SF_m)[e_{m+1,m+1}]}{\left(\sum\limits_{i=0}^{n-m}(-1)^ip_{n-m-i}(\Tc_{m,m}')e_{m+1,m+1}^{2i}\right)}\cong A^*(\SF_{m+1}).
\]
The induction assumption provides the description for the coefficients of this polynomial algebra that is
\[
A^*(\SF_m)\cong \bigslant{A^*(pt)[e_{1,m},\ldots,e_{m,m}]}{I_{2m,2n+1}}.
\]
Note that $e_{i,m}=e_{i,m+1}$ for $i\le m$. Thus it is sufficient to show that 
\[
\sum_{i=0}^{n-m}(-1)p^i_{n-m-i}(\Tc_{m,m}')e_{m+1,m+1}^{2i}=(-1)^{n-m}h_{n-m}(e_{1,m+1}^2,e_{2,m+1}^2,\ldots,e_{m+1,m+1}^2).
\]
Applying Lemma~\ref{lem_borelcompletepoly} we obtain 
\[
p_{n-m-i}(\Tc_{m,m}')=(-1)^{n-m-i}h_{n-m-i}(e_{1,m}^2,e_{2,m}^2,\ldots,e_{m,m}^2).
\]
The claim follows from the well-known identity 
\[
h_k(x_1,\ldots,x_l)=\sum_{i=0}^{k}h_i(x_1,\ldots,x_{l-1})x_l^{k-i}. \qedhere
\]
\end{proof}
\begin{rem}
The relations in the proposition arise from the comparison of the different descriptions for the top Pontryagin class of $\Tc_{i,m}'$:
\[
(e_{i+1}\ldots e_{m+1})^2=e(\Tc_{i,m}')^2=p_{n-i}(\Tc_{i}')=(-1)^{n-i}h_{n-i}(e_1^2,\ldots, e_i^2).
\]
In the even case there is another one relation expressing triviality of the Euler class of the bundle $(\oplus_{i=1}^m \Tc_i)\oplus \Tc_{m+1}$.
\end{rem}

There is another description for the ideals $I_{2m,2n}$ and $I_{2m,2n+1}$.
\begin{lemma} 
\label{lem_partialRel}
For the above ideals $I_{2m,2n}$ and $I_{2m,2n+1}$ we have
\begin{align*}
I_{2m,2n}=\big(&e_1e_2\ldots e_me_m', (-1)^{n-m+1}e_m'^2+h_{n-m}(e_1^2,e_2^2,\ldots,e_m^2), \\
 & h_{n-m+1}(e_1^2,\ldots,e_m^2), h_{n-m+2}(e_1^2,\ldots,e_m^2),\ldots, h_{n-1}(e_1^2,\ldots,e_m^2) \big),\\
I_{2m,2n+1}=\big(&h_{n-m+1}(e_1^2,\ldots,e_m^2), h_{n-m+2}(e_1^2,\ldots,e_m^2),\ldots, h_{n}(e_1^2,\ldots,e_m^2) \big).
\end{align*}
\end{lemma}
\begin{proof}
This equalities follow from the obvious identity 
\[
h_i(x_1,x_2,\ldots,x_n)=h_i(x_1,x_2,\ldots,x_{n-1})+x_nh_{i-1}(x_1,x_2,\ldots,x_n).\qedhere
\]
\end{proof}

\begin{rem}
The vanishing of the polynomials $h_{i}(e_1^2,\ldots,e_m^2)$ corresponds to the vanishing of the Pontryagin classes $p_i(\Tc_{m+1}), i>n-m$ for the tautological bundle $\Tc_{m+1}$ over $\SF(2,4,\ldots,2m,2n+1)$.
\end{rem}
\begin{rem}
\label{rem_SF}
Consider the case of the maximal $SL_2$ flag variety, i.e. $\SF(2n)$ and $\SF(2n+1)$. Investigation of the above relations yields that the cohomology of these varieties coincide with the algebras of coinvariants for the Weyl groups $W(D_n)$ and $W(B_n)$ respectively. In other words, there are isomorphisms
\begin{align*}
 A^*(\SF(2n)) \cong \bigslant{A^*(pt)[e_1,e_2,...,e_{n}]}{\big( s_1,s_2,\dots,s_{n-1},t\big)} ,\\
A^*(\SF(2n+1)) \cong \bigslant{A^*(pt)[e_1,e_2,...,e_{n}]}{\big( s_1,s_2,\dots,s_n\big)}
\end{align*}
where $s_i=\sigma_i(e_1^2,e_2^2,\ldots,e_n^2)$ for the elementary symmetric polynomials $\sigma_i$ and $t=e_1e_2\ldots e_n$.
\end{rem}

\section{Symmetric polynomials.}
In this section we deal with the polynomials invariant under the action of the Weyl group $W(B_n)$ or $W(D_n)$ and obtain certain spanning sets for the polynomial rings. Our method is an adaptation of the one used in \cite[\S~10,~Proposition~3]{Ful}. 

Consider $\mathbb{Z}^{n}$ and fix a usual basis $\{e_1,\dots,e_n\}$. Let
$$
W(B_n)=\{\phi\in Aut(\mathbb{Z}^n)\,|\, \phi(e_i)=\pm e_j\}
$$
be the Weyl group of the root system $B_n$ and let
$$
W(D_n)=\{\phi\in Aut(\mathbb{Z}^n)\,|\, \phi(e_i)=(-1)^{k_i}e_j,\, (-1)^{\sum k_i}=1\}
$$
be the Weyl group of the root system $D_n$. Identifying $R=\mathbb{Z}[e_1,\dots,e_n]$ with the symmetric algebra $Sym^*((\mathbb{Z}^n)^\vee)$ in a usual way, we obtain the actions of these Weyl groups on $R$. Let $R_B=R^{W(B_n)}$ and $R_D=R^{W(D_n)}$ be the algebras of invariants.

For the elementary polynomials $\sigma_i\in \mathbb{Z}[x_1,\dots,x_n]$ consider
$$
s_i=\sigma_i(e_1^2,\dots,e_n^2),\quad t=\sigma_n(e_1,\dots,e_n).
$$
One can easily check that $R_B=\mathbb{Z}[s_1,\dots, s_n]$ and $R_D=\mathbb{Z}[s_1,\dots, s_{n-1},t]$.

In order to compute spanning sets for $R$ over $R_B$ and $R_D$ we need "decreasing degree" equalities provided by the following lemma.
\begin{lemma}
\label{lem_lower}
There exist polynomials $g_{i}, h_i\in R$ such that
$$
e_1^{2n}=\sum_{i=1}^{n} g_is_i, \quad e_1^{2n-1}=\sum_{i=1}^{n-1}h_is_i+h_nt.
$$
One may assume that $g_i$ and $h_i$ are homogeneous of appropriate degrees, i.e. $\deg g_i=2n-2i$, $\deg h_i=2n-2i-1$ and $\deg h_n=n-1$.
\end{lemma}
\begin{proof}
Let $I_B=\big( s_1,\dots,s_n\big)$ and $I_D=\big( s_1,\dots,s_{n-1},t \big)$ be the ideals generated by the homogeneous invariant polynomials of positive degree. We need to show that $e_1^{2n}\in I_B$ and $e_1^{2n-1}\in I_D$.  Set $S_B=R/I_B$, $S_D=R/I_D$.

Consider $S_B\left[\left[x\right]\right]$. Since all the $s_i$ belong to $I_B$ we have
$$
(1-\bar{e}_1^2x)(1-\bar{e}_2^2x){\dots}(1-\bar{e}_n^2x)=1,
$$
hence
$$
(1-\bar{e}_2^2x)(1-\bar{e}_3^2x){\dots}(1-\bar{e}_n^2x)=1+\bar{e}_1^2x+\bar{e}_1^4x^2+{\dots}
$$
Comparing the coefficients at $x^n$ we obtain $\bar{e}_1^{2n}=0$, thus $e_1^{2n}\in I_B$.

Consider $S_D\left[\left[x\right]\right]$. As above, we have
$$
(1-\bar{e}_1^2x^2)(1-\bar{e}_2^2x^2){\dots}(1-\bar{e}_n^2x^2)=1,
$$
hence
$$
(1+\bar{e}_1x)(1-\bar{e}_2^2x^2)(1-\bar{e}_3^2x^2){\dots}(1-\bar{e}_n^2x^2)=1+\bar{e}_1x+\bar{e}_1^2x^2+{\dots}
$$
Comparing the coefficients at $x^{2n-1}$ we obtain
$$
\bar{e}_1^{2n-1}=(-1)^{n-1}\bar{e}_1\bar{e}_2^2{\dots}\bar{e}_n^2= (-1)^{n-1}\bar{t}\bar{e}_2\bar{e}_3{\dots}\bar{e}_n=0,
$$
thus $e_1^{2n-1}\in I_D$.
\end{proof}

\begin{proposition}
\label{prop_spanning}
In the above notation we have the following spanning sets:
\begin{enumerate}
\item
$
\mathcal{B}_1=\left\{\left.e_1^{m_1}e_2^{m_2}\dots e_n^{m_n}\,\right|\, 0\le m_i\le 2n-2i+1\right\}
$\\
spans $R$ over $R_B$.
\item
$
\mathcal{B}_2=\left\{u_1u_2\dots u_{n-1} \,\left|\, u_i=
\left[\begin{array}{l}e_i^{m_i},\, 0\le m_i\le 2n-2i \\  e_{i+1}e_{i+2}\dots e_n \end{array}\right. \right.\right\}
$\\
spans $R$ over $R_D$.
\end{enumerate}
\end{proposition}
\begin{proof}
In both cases proceed by induction on $n$. The base case of $n=1$ is clear. Denote by $\mathcal{B}_1'$ and $\mathcal{B}_2'$ the spanning sets in $R'=\mathbb{Z}[e_2,{\dots},e_n]$ and let $s_i',t'\in R'$ be the corresponding invariant polynomials. Note that $s_i=e_1^2s_{i-1}'+s_i'$ and $t=e_1t'$.


Suppose that one can not express some monomial as a linear combination of $\mathcal{B}_1$ (or $\mathcal{B}_2$) with $R_B$-coefficients (or $R_D$-coefficients). Consider among these monomials the ones of minimal degree and choose among them a monomial with the largest degree at $e_1$, denote it by $f=e_1^{k_1}e_2^{k_2}{\dots}e_n^{k_n}\in R$.

(1) In case of $k_1\ge 2n$ we can use Lemma~\ref{lem_lower} and substitute $\sum g_i s_i$ for $e_1^{2n}$ obtaining
$$
f=\sum s_ig_ie_1^{k_1-2n}e_2^{k_2}{\dots}e_n^{k_n}
$$
with $\deg g_ie_1^{k_1-2n}e_2^{k_2}{\dots}e_n^{k_n} < \deg f$, so the right-hand side is an $R_B$-linear combination of $\mathcal{B}_1$. Now suppose that $k_1<2n$. By the induction we have
$$
e_2^{k_2}e_3^{k_3}{\dots}e_n^{k_n}=\sum \alpha_j(s_1',\dots,s_{n-1}')b_j'
$$
for some $b_j'\in \mathcal{B}_1'$ and $\alpha_j\in \mathbb{Z}[x_1,\dots,x_{n-1}]$. We can assume that all the summands at the right-hand side are homogeneous of total degree $k_2+{\dots}+k_n$. Since $s_i=e_1^2s_{i-1}'+s_i'$ one has
$$
\alpha_j(s_1,\dots,s_{n-1})= \alpha_j\left(s_1',\dots,s_{n-1}'\right)+\sum_{l>0}e_1^l\beta_{jl}
$$
for some $\beta_{jl} \in R_B'$. Thus we obtain
$$
f=\sum_j e_1^{k_1}\alpha_j(s_1,\dots,s_{n-1})b_j'-\sum_{j,l} e_1^{k_1+l}\beta_{jl} b_j'.
$$
Note that $e_1^{k_1}b_j'\in\mathcal{B}_1$, so the first sum is an $R_B$-linear combination of the monomials from the spanning set. The second sum consists of monomials of degree $\deg f$ and degrees at $e_1$ of these monomials are greater then $k_1$, so by the choice of $f$  it is an $R_B$-linear combination of the monomials from the spanning set. Hence $f$ is an $R_B$-linear combination of $\mathcal{B}_1$, contradicting the assumption.

(2) As above, in case of $k_1\ge 2n-1$ we can use Lemma~\ref{lem_lower} and lower the total degree, so suppose that $k_1<2n -1$. By induction we have
$$
e_2^{k_2}e_3^{k_3}{\dots}e_n^{k_n}=\sum \alpha_j(s_1',\dots,s_{n-2}',t')b_j'
$$
for some $b_j'\in \mathcal{B}_2'$ and $\alpha_j\in \mathbb{Z}[x_1,\dots,x_{n-1}]$. We may assume that all the summands on the right-hand side are homogeneous of degree $k_2+k_3+\ldots+k_n$. One has $t'^2=s_{n-1}'$ hence
$$
\alpha_j(s_1',\dots,s_{n-2}',t)=\widetilde{\alpha}_j(s_1',\dots,s_{n-2}',s_{n-1}')+t'\widehat{\alpha}_j(s_1',\dots,s_{n-2}',s_{n-1}').
$$
As above, we can substitute $s_i$ into $\widetilde{\alpha}_j$ and $\widehat{\alpha}_j$ and obtain some $\widetilde{\beta}_{jl},\widehat{\beta}_{jl}\in R_D'$. Thus we have
$$
f=\sum_j e_1^{k_1}\widetilde{\alpha}_j(s_1,\dots,s_{n-1})b_j'+\sum_j e_1^{k_1}\widehat{\alpha}_j(s_1,\dots,s_{n-1})t'b_j'-
$$
$$
-\sum_{j,l} e_1^{k_1+l}\widetilde{\beta}_{jl} b_j'-\sum_{j,l} e_1^{k_1+l}\widehat{\beta}_{jl}t' b_j'.
$$
In the first sum we have $e_1^{k_1}b_j'\in \mathcal{B}_2$. One has $t'b_j'\in \mathcal{B}_2$, so in case of $k_1=0$ the second sum is a linear combination of the elements from the spanning set, otherwise, if $k_1\ge 1$, one has $t=e_1t'$ and $\deg (e_1^{k_1-1}\widehat{\alpha}_j(s_1,\dots,s_{n-1})b_j') < \deg f$, so in both cases the second sum is an $R_D$-linear combination of $\mathcal{B}_2$. The third and fourth sums are dealt with like the second one in (1), the monomials have the same degree as $f$ but the degrees at $e_1$ are greater. Thus we obtain that $f$ is an $R_D$-linear combination of $\mathcal{B}_2$ contradicting the assumption.

\end{proof}

\section{Cohomology of the special linear Grassmannians and $BSL_n$.}

Now we are ready to compute the cohomology of the special linear Grassmannians. Recall that $h_i(x_1,x_2,\ldots,x_n)=g_i(\sigma_1,\sigma_2,\ldots,\sigma_n)$ for a certain polynomial $g_i\in \mathbb{Z}[y_1,y_2,\ldots,y_n]$.

\begin{theorem}
\label{thm_SGr}
For the special linear Grassmannians we have the following isomorphisms of $A^*(pt)$-algebras.
\begin{enumerate}
\item
$
\phi_1\colon \bigslant{A^*(pt)[p_1,p_2,...,p_{m},e,e']}{J_{2m,2n}} \xrightarrow{\simeq} A^*(SGr(2m,2n)),
$
where 
\begin{align*}
J_{2m,2n}=\big(&ee', e^2-p_m, (-1)^{n-m+1}e'^2+g_{n-m}(p_1,p_2,\ldots,p_m), \\
 & g_{n-m+1}(p_1,p_2\ldots,p_m), g_{n-m+2}(p_1,p_2,\ldots,p_m), \ldots,\\
 & g_{n-1}(p_1,p_2,\ldots,p_m) \big)
\end{align*}
and the isomorphism is induced by $\phi_1(p_i)=p_i(\Tc_1)$, $\phi_1(e)=e(\Tc_1)$ and $\phi_1(e')=e(\Tc_2)$.
\item
$
\phi_2\colon \bigslant{A^*(pt)[p_1,p_2,...,p_{m},e]}{J_{2m,2n+1}} \xrightarrow{\simeq} A^*(SGr(2m,2n+1)),
$
where 
\begin{align*}
J_{2m,2n+1}=\big(&e^2-p_m, g_{n-m+1}(p_1,p_2,\ldots,p_m), \\
&g_{n-m+2}(p_1,p_2,\ldots,p_m), \ldots, g_{n}(p_1,p_2,\ldots,p_m)\big)
\end{align*}
and the isomorphism is induced by $\phi_2(p_i)=p_i(\Tc_1)$ and $\phi_2(e)=e(\Tc_1)$.
\end{enumerate}
\end{theorem}
\begin{proof}
\textbf{(1)} Consider the special linear flag variety $p\colon \SF(\Tc_1)\to SGr(2m,2n)$. The homomorphism $p^A$ is injective by Theorem~\ref{thm_split}. There is an isomorphism
\[
\SF(\Tc_1)\cong \SF(2,4,\ldots,2m,2n).
\]
Denote this variety by $\SF$. Proposition~\ref{prop_partialFlag} yields that there is an injection
\[
p^A\colon A^*(SGr(2m,2n))\to \bigslant{A^*(pt)[e_1,\ldots,e_m,e'_m]}{I_{2m,2n}}.
\]
We have $p^A(e(\Tc_1))=e_1e_2\ldots e_m, p^A(e(\Tc_2))=e_m'$ and by Lemma~\ref{lemm_borel_cl} and multiplicativity of total Pontryagin classes  $p^A(p_i(\Tc_1))=\sigma_i(e_1^2,e_2^2,\ldots,e_m^2)$. From now on we omit $p^A$ and regard $A^*(SGr(2m,2n))$ as a subalgebra of $A^*(\SF)$. Lemma~\ref{lem_partialRel} shows that $J_{2m,2n} \subset I_{2m,2n}=J_{2m,2n}A^*(pt)[e_1,\ldots,e_m,e_m']$, moreover, 
\[
J_{2m,2n}=I_{2m,2n}\cap A^*(pt)[p_1,\ldots,p_m,e,e']
\]
since by Proposition~\ref{prop_spanning} the algebra $A^*(pt)[e_1,\ldots,e_m,e'_m]$ is a free module over $A^*(pt)[p_1,\ldots,p_m,e,e']$.

Hence there exists the announced map
\[
\phi_1\colon \bigslant{A^*(pt)[p_1,p_2,...,p_{m},e,e']}{J_{2m,2n}} \xrightarrow{} A^*(SGr(2m,2n))\subset A^*(\SF)
\]
with $\phi_1(p_i)=p_i(\Tc_1)=\sigma_i(e_1^2,e_2^2,\ldots,e_m^2)$, $\phi_1(e)=e(\Tc_1)=e_1e_2\ldots e_m$ and $\phi_1(e')=e(\Tc_2)=e_m'$ and it is injective.

Applying Theorem~\ref{thm_split} we obtain that the set
$$
\mathcal{B}
=\left\{u_1u_2\dots u_{m-1} \,\left|\, u_i=
\left[\begin{array}{l}e_i^{m_i},\, 0\le m_i\le 2n-2i \\  e_{i+1}e_{i+2}\dots e_{m} \end{array}\right. \right.\right\},$$
forms a basis of $A^*(\SF)$ over $A^*(SGr(2m,2n))$. Note that by the same theorem $A^*(\SF)$ is generated as an $A^*(pt)$-algebra by $e_1,e_2\dots,e_m,e_m'$, thus by Proposition~\ref{prop_spanning} we know that $\mathcal{B}$ spans $A^*(\SF)$ over the algebra 
\[
Im (\phi_1)=A^*(pt)[\phi_1(p_1),\phi_1(p_2),\dots,\phi_1(p_{m-1}),\phi_1(e),\phi_1(e')],
\]
hence $Im(\phi_1)=A^*(SGr(2m,2n))$ and $\phi_1$ is surjective.

\textbf{(2)} could be obtained via the similar reasoning.
\end{proof}

\begin{rem}
It seems that there is no good description for $A^*(SGr(2m+1,2n))$ in terms of the Euler and Pontryagin characteristic classes. For instance, consider the simplest example $SGr(1,2)\cong\A^2-\{0\}$. It is isomorphic to the unpointed motivic sphere $S^{3,2}$, so $A^{*}(SGr(1,2))\cong A^{*}(pt)\oplus A^{*-1}(pt)$, but we do not have an appropriate nontrivial special linear bundle over $A^*(\A^2-\{0\})$ to take the characteristic class. Another complication comes from the fact that the grading shift is odd whereas our characteristic classes lie in the even degrees.
\end{rem}

Now we turn to the computation of the cohomology rings of the classifying spaces
$$
BSL_n=\varinjlim\limits_{m\in \mathbb{N}} SGr(n,m).
$$
The case of $BSL_{2n}$ easily follows from Theorem~\ref{thm_SGr}. In order to compute the cohomology of $BSL_{2n+1}$ we will use a certain Gysin sequence relating $A^*(BSL_{2n+1})$ to $A^*(BSL_{2n})$.

Recall that $A^*$ is constructed from a representable cohomology theory. In this setting we have the following proposition relating the cohomology groups of a limit space to the limit of the cohomology groups \cite[Lemma~A.5.10]{PPR2}.
\begin{proposition}
\label{prop_lim1}
For any sequence of motivic spaces $X_1\xrightarrow{i_1} X_2\xrightarrow{i_2} X_3\xrightarrow{i_3}\dots$ and any $p$ we have an exact sequence of abelian groups
$$
0\to{\varprojlim}^1 A^{p-1}(X_k)\to A^p(\varinjlim X_k) \to \varprojlim A^p(X_k)\to 0.
$$
\end{proposition}
As usual, the $\lim^1$ term vanishes whenever the Mittag-Leffler condition is satisfied, i.e. if for every $i$ there exists some $k$ such that for every $j\ge k$ one has $Im(A^{*}(X_j)\to A^{*}(X_i))=Im(A^{*}(X_k)\to A^{*}(X_i))$.

Consider the sequence of embeddings
$$
{\dots}\to SGr(2n,2m+1)\xrightarrow{i_{2m+1}} SGr(2n,2m+3) \to {\dots}
$$
By Theorem~\ref{thm_SGr} we know that $i_{2m+1}^A$ is surjective hence
$$
A^p(BSL_{2n})\cong\varprojlim A^p(SGr(2n,2m+1))\cong \varprojlim A^p(SGr(2n,m)).
$$
The sequence of the tautological special linear bundles $\Tc_1$ over $SGr(2n,m)$ gives rise to a bundle $\Tc$ over $BSL_{2n}$. We have a sequence of embeddings of the Thom spaces
$$
{\dots}\to Th(\Tc_1(2n,2m+1)) \xrightarrow{j_{2m+1}} Th(\Tc_1(2n,2m+3)) \to {\dots}
$$
where $\Tc_1(i,j)$ is the first tautological special linear bundle over $SGr(i,j)$. Since all the considered morphisms $\Tc_1(2n,k)\to \Tc_1(2n,l)$ are inclusions there is a canonical isomorphism $\Tc/(\Tc-BSL_{2n})=Th(\Tc)\cong \varinjlim \Tc_1(2n,m)$. For every $k$ we have an isomorphism
$$
A^{*-2n}(SGr(2n,k))\xrightarrow{\cup th(\Tc_1(2n,k))} A^*(Th(\Tc_1(2n,k))),
$$
so $j_{2m+1}^A$ are surjective as well as $i_{2m+1}$ and
$$
A^p(Th(\Tc))\cong\varprojlim A^p(\Tc_1(2n,m)).
$$
\begin{definition}
Let $\Tc$ be the tautological bundle over $BSL_{2n}$. Denote by $b_i(\Tc), e(\Tc)\in A^*(BSL_{2n})$ and $th(\Tc)\in A^*(Th(\Tc))$ the elements corresponding to the sequences of the classes of the tautological bundles,
$$
p_i(\Tc)=(\dots,p_i(\Tc_1(2n,m)),p_i(\Tc_1(2n,m+1)),\dots),
$$
$$
e(\Tc)=(\dots,e(\Tc_1(2n,m)),e(\Tc_1(2n,m+1)),\dots),
$$
$$
th(\Tc)=(\dots,th(\Tc_1(2n,m)),th(\Tc_1(2n,m+1)),\dots),
$$
with $\Tc_1(2n,m)$ being the tautological special linear bundle over $SGr(2n,m)$.
\end{definition}

The above considerations show that we have a Gysin sequence for the tautological bundle over the classifying space $BSL_{2n}$.
\begin{lemma}
\label{lemm_gysinBSL}
Let $\Tc$ be the tautological bundle over $BSL_{2n}$. Then there exists a long exact sequence
$$
{}\to A^{*-2n}(BSL_{2n})\xrightarrow{\cup e(\Tc)} A^*(BSL_{2n}) \xrightarrow{j^A} A^*(BSL_{2n-1})\xrightarrow{\partial}{}
$$
\end{lemma}
\begin{proof}
For the zero section inclusion of motivic spaces $BSL_{2n}\to \Tc$ we have the following long exact sequence.
$$
{\dots}\to A^{*}(Th(\Tc))\xrightarrow{} A^*(\Tc) \xrightarrow{} A^*(\Tc^0)\xrightarrow{\partial}{\dots}
$$
The isomorphisms
$$
A^{*-2n}(SGr(2n,k))\xrightarrow{\cup th(\Tc_1(2n,k))} Th(\Tc_1(2n,k)),
$$
induce an isomorphism $A^{*-2n}(BSL_{2n})\xrightarrow{\cup th(\Tc)} A^*(Th(\Tc))$, so we can substitute $A^{*-2n}(BSL_{2n})$ for the first term in the above sequence. Using homotopy invariance we exchange $\Tc$ for $BSL_{2n}$. By the definition of $e(\Tc)$ the first arrow represents the cup product $\cup e(\Tc)$.

We have isomorphisms
$$
\Tc^0\cong\varinjlim SGr(1,2n-1,m)\cong \varinjlim SGr(2n-1,1,m).
$$
The sequence of projections
$$
\xymatrix{
{\dots}\ar[r] & SGr(2n-1,1,m) \ar[r]\ar[d] & SGr(2n-1,1,m+1) \ar[r]\ar[d] & {\dots} \\
{\dots}\ar[r] & SGr(2n-1,m+1) \ar[r] & SGr(2n-1,m+2) \ar[r]& {\dots}
}
$$
induces a morphism $\Tc^0\xrightarrow{r} BSL_{2n-1}$. Note that
$$
SGr(2n-1,1,m)\cong \Tc_2(2n-1,m+1)^0,
$$
and $\Tc^0$ is an $\A^\infty-\{0\}$-bundle over $BSL_{2n-1}$, so by \cite[Section~4, Proposition~2.3]{MV} $r$ is an isomorphism in the homotopy category and we can substitute $A^*(BSL_{2n-1})$ for the third term in the long exact sequence.
\end{proof}


\begin{definition}
For a graded ring $R^*$ let $R^*\left[\left[t\right]\right]_h$ be the homogeneous power series ring, i.e. a graded ring with
$$
R^*\left[\left[t\right]\right]_h^k=\left\{\sum a_{i} t^i\,|\, \deg a_{i}+i\deg t=k\right\}.
$$
Note that $R^*\left[\left[t\right]\right]_h=\varprojlim R^*[t]/t^n$, where the limit is taken in the category of graded algebras.
We set $R^*\left[\left[t_1,\dots,t_n\right]\right]_h=R^*\left[\left[t_1,\dots,t_{n-1}\right]\right]_h\left[\left[t_n\right]\right]_h$.
\end{definition}

%
%
%

\begin{theorem}
\label{thm_BSL}
For $\deg e=2n, \deg p_i =4i$ we have isomorphisms
$$
A^*(pt)\left[\left[p_1,\dots,p_{n-1},e\right]\right]_h\xrightarrow{\simeq} A^*(BSL_{2n}),
$$
$$
A^*(pt)\left[\left[p_1,\dots,p_{n}\right]\right]_h\xrightarrow{\simeq} A^*(BSL_{2n+1}).
$$
\end{theorem}
\begin{proof}
The case of $BSL_{2n}$ follows from Theorem~\ref{thm_SGr} and Proposition~\ref{prop_lim1}, since for the sequence
$$
{\dots}\to SGr(2n,2m+1) \xrightarrow{i_{2m+1}} SGr(2n,2m+3) \to {\dots}
$$
the pullbacks $i_{2m+1}^A$ are surjective and $\lim^1$ vanishes yielding
\begin{multline*}
A^*(BSL_2)\cong \varprojlim A^*(SGr(2n,2m+1)) = \\
=\varprojlim A^*(pt)[p_1,p_2,\ldots,p_n,e]/J_{2n,2m+1}
=A^*(pt)\left[\left[p_1,\dots,p_{n-1},e\right]\right]_h.
\end{multline*}

For the odd case consider the Gysin sequence from Lemma~\ref{lemm_gysinBSL} for $BSL_{2n+2}$.
By the above calculations $e(\Tc)$ is not a zero divisor, so the the map $\cup e(\Tc)$ is injective and we have a short exact sequence
$$
0\to A^{*-2n-2}(BSL_{2n+2})\xrightarrow{\cup e(\Tc)} A^*(BSL_{2n+2})\to A^*(BSL_{2n+1})\to 0.
$$
Identifying $A^*(BSL_{2n+2})$ with the homogeneous power series and killing $e$ we obtain the desired result.
\end{proof}

\begin{rem}
Another way to compute $A^*(BSL_{2n+1})$ is to use the calculation for $A^*(SGr(2n+1,2m+1))\cong A^*(SGr(2m-2n,2m+1))$. The Euler classes are unstable, so the image 
\[
Im(A^*(SGr(2n+1,2m+3))\to A^*(SGr(2n+1,2m+1)))
\]
 is generated by the Pontryagin classes $p_i(\Tc_2)$ and $\lim^1$ vanishes. One could express $p_i(\Tc_2)$ in terms of $p_i(\Tc_1)$, obtaining the desired result.
\end{rem}


\end{document}